\documentclass[twoside,a4paper,11pt]{amsart}

\usepackage[symbol]{footmisc}
\usepackage{epigraph}
\usepackage{geometry}
\usepackage{amsmath}
\usepackage{amssymb}
\usepackage{amsthm}
\usepackage{mathrsfs}
\usepackage{mathtools}
\usepackage{esint}
\usepackage{graphicx}
\usepackage{color}
\usepackage{hyperref}
\usepackage{tikz}
\usepackage{float}
\usetikzlibrary{patterns}
\usepackage[english]{babel}

\newcommand{\N}{\mathbb{N}}
\newcommand{\Q}{\mathbb{Q}}
\newcommand{\R}{\mathbb{R}}
\renewcommand{\S}{\mathbb{S}}

\newcommand{\cA}{\mathcal{A}}

\newcommand{\cC}{\mathcal{C}}

\newcommand{\cE}{\mathcal{E}}
\newcommand{\cF}{\mathcal{F}}

\newcommand{\cH}{\mathcal{H}}

\newcommand{\cL}{\mathcal{L}}
\newcommand{\cN}{\mathcal{N}}

\newcommand{\cM}{\mathcal{M}}

\newcommand{\fP}{\mathfrak{P}}

\newcommand{\bv}{\mathbf v}

\newcommand{\tE}{\tilde{E}}

\newcommand{\tV}{\tilde{V}}

\newcommand{\tGa}{\tilde{\Gamma}}

\newcommand{\diam}{\mbox{\rm diam}}

\newcommand{\supp}{\mbox{\rm supp}}

\renewcommand{\div}{\mbox{\rm div}}

\newcommand{\ind}{\mbox{\rm Ind}}

\newcommand{\wsc}{\overset{\star}{\rightharpoonup}}

\newcommand{\bcup}{\bigcup}
\newcommand{\epty}{\emptyset}
\newcommand{\sen}{\sin}
\newcommand{\sm}{\setminus}
\newcommand{\lgl}{\langle}
\newcommand{\rgl}{\rangle}
\newcommand{\pa}{\partial}
\newcommand{\con}{\subset}

\newcommand{\res}{
	\,\raisebox{-.127ex}{\reflectbox{\rotatebox[origin=br]{-90}{$\lnot$}}}\,
} %simbolo di restrizione di misura

\newcommand{\beq}{\begin{equation}}
\newcommand{\eeq}{\end{equation}}
\newcommand{\beqs}{\begin{equation*}}
\newcommand{\eeqs}{\end{equation*}}
\newcommand{\ep}{\varepsilon} %epsilon più carina
\newcommand{\ga}{\gamma}
\newcommand{\be}{\beta}
\newcommand{\al}{\alpha}
\newcommand{\de}{\delta}

\newcommand{\la}{\lambda}

\newcommand{\om}{\omega}
\newcommand{\ro}{\rho}
\newcommand{\si}{\sigma}
\newcommand{\te}{\theta}
\newcommand{\De}{\Delta}
\newcommand{\Ga}{\Gamma}
\newcommand{\La}{\Lambda}
\newcommand{\Si}{\Sigma}

\newcommand{\Om}{\Omega}
\newcommand{\vp}{\varphi}

\newcommand{\fp}{\cF_p}
\newcommand{\rfp}{\overline{\cF_p}}
\newcommand{\rfd}{\overline{\cF_2}}
\newcommand{\flp}{\cF_{\lambda,p}}
\newcommand{\rflp}{\overline{\cF_{\lambda,p}}}

\theoremstyle{plain}
\newtheorem{thm}{Theorem}[section] 

\theoremstyle{plain}

\theoremstyle{plain}
\newtheorem{prop}[thm]{Proposition}

\theoremstyle{plain}
\newtheorem{lemma}[thm]{Lemma}

\theoremstyle{plain}
\newtheorem{cor}[thm]{Corollary}

\theoremstyle{definition}
\newtheorem{defn}[thm]{Definition} % definition numbers are dependent on theorem numbers

\theoremstyle{definition}
\newtheorem{remark}[thm]{Remark}

\theoremstyle{definition}
\newtheorem{example}[thm]{Example}

\title[The $p$-elastic energy of planar sets]{A VARIFOLD PERSPECTIVE ON THE $p$-ELASTIC ENERGY\\ OF PLANAR SETS}

\author{Marco Pozzetta}
\address{Dipartimento di Matematica, Universit\`{a} di Pisa, Largo Bruno Pontecorvo 5, 56127 Pisa, Italy}
\email{pozzetta@mail.dm.unipi.it}
\date{\today}

\begin{document}

\begin{abstract}
	Under suitable regularity assumptions, the $p$-elastic energy of a planar set $E\con\R^2$ is defined to be
	\beqs
		\fp(E)=\int_{\pa E} 1 + |k_{\pa E}|^p \,\, d\cH^1,
	\eeqs
	where $k_{\pa E}$ is the curvature of the boundary $\pa E$. In this work we use a varifold approach to investigate this energy, that can be well defined on varifolds with curvature. First we show new tools for the study of $1$-dimensional curvature varifolds, such as existence and uniform bounds on the density of varifolds with finite elastic energy. Then we characterize a new notion of $L^1$-relaxation of this energy by extending the definition of regular sets by an intrinsic varifold perspective, also comparing this relaxation with the classical one of \cite{BeMu04}, \cite{BeMu07}. Finally we discuss an application to the inpainting problem, examples and qualitative properties of sets with finite relaxed energy.
\end{abstract}
\maketitle

\vspace{-12mm}
\tableofcontents	
\vspace{-0.2cm}

\vspace{0.2cm}

%%%%%%%%%%%%%%%%%%%%%%%%%%%%%%%%%%%%%%%%%%%%%%%%%%%%%%%%%%%%%%%%%%%%%%%%%%%%%%

\section{Introduction}
%\addcontentsline{toc}{section}{Introduction}

\noindent Consider $p\in[1,\infty)$. Let us denote by $S^1$ the interval $[0,2\pi]$ with the identification $0\sim2\pi$. For an immersion $\ga:S^1\to\R^2$ such that $\ga\in W^{2,p}(S^1)$ we can define the functional
\beq \label{def1}
	\cE_p(\ga)=\int_0^{2\pi} |k_\ga|^p|\ga'|\,dt,
\eeq
and the $p$-elastic energy
\beq \label{def2}
	\cF_p(\ga)=L(\ga)+\cE_p(\ga),
\eeq
where $L(\ga)$ denotes the length of $\ga$.\\

\noindent In this work we want to study the elastic properties of the boundaries of measurable sets in $\R^2$. Our first purpose is to give a new definition of the sets which are enough regular for having finite $p$-elastic energy. We want such definition to be intrinsically dependent on the given set, using immersions of curves only as a tool for the calculation of the energy.\\

\noindent In order to study the functionals defined in \eqref{def1} and \eqref{def2} one would classically call regular set a set $E$ with a boundary of class $C^2$, i.e. a set $E$ whose $\pa E$ is the image of closed injective immersions $\ga:S^1\to\R^2$ of class $C^2$. This would be a possible definition of set with finite classic $p$-elastic energy, and it is the definition considered in \cite{BeDaPa93}, \cite{BeMu04} and \cite{BeMu07} indeed. But with this classical definition it turns out that sets like the one in Fig. \ref{figotto} not only have infinite energy, but they also have infinite relaxed energy (calculated with respect to the $L^1$-convergence of sets, see \cite{BeMu04}).

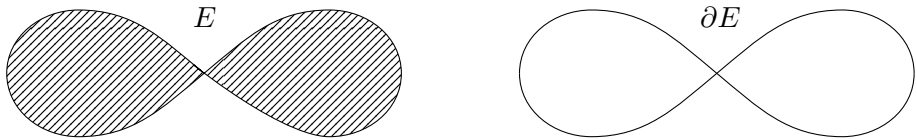
\begin{figure}[h]
	\begin{center}
		\begin{tikzpicture}[scale=1.5, rotate=90]
		\filldraw[fill=white, pattern=north east lines]
		(0,1.73)to[out= 0,in=90, looseness=1] (0.56,1.1)to[out= -90,in=50, looseness=1] (0,0)to[out=-50, in=90, looseness=1](0.56,-1.1)to[out=270, in=0, looseness=1](0,-1.73)to[out=180, in=270, looseness=1](-0.56,-1.1)to[out=90, in=50, looseness=1](0,0)to[out=130, in=270, looseness=1](-0.56,1.1)to[out=90, in=180, looseness=1](0,1.73);
		\path[font=\normalsize]
		(0.5,-0.2)node[left]{$E$};
		\end{tikzpicture}\qquad\qquad
		\begin{tikzpicture}[scale=1.5, rotate=90]
		\draw
		(0,1.73)to[out= 0,in=90, looseness=1] (0.56,1.1)
		(0.56,1.1)to[out= -90,in=50, looseness=1] (0,0);
		\draw
		(0,1.73)to[out= 180,in=90, looseness=1] (-0.56,1.1)
		(-0.56,1.1)to[out= -90,in=130, looseness=1] (0,0);
		\draw[rotate=180]
		(0,1.73)to[out= 0,in=90, looseness=1] (0.56,1.1)
		(0.56,1.1)to[out= -90,in=50, looseness=1] (0,0);
		\draw[rotate=180]
		(0,1.73)to[out= 180,in=90, looseness=1] (-0.56,1.1)
		(-0.56,1.1)to[out= -90,in=130, looseness=1] (0,0);
		\path[font=\normalsize]
		(0.5,-0.3)node[left]{$\pa E$};
		\end{tikzpicture}
	\end{center}
	\caption{A set of finite perimeter $E$ with boundary $\pa E$ that can be parametrized by a smooth non-injective immersion.}\label{figotto}
\end{figure}

\noindent However functionals \eqref{def1} and \eqref{def2} are very well defined on immersions which are not necessarily injective. Also for many applications one would like to consider sets like the one in Fig. \ref{figotto} as regular sets, or at least as sets with finite relaxed energy (applications will also be discussed below). A good definition of regular elastic set, i.e. a definition of set with finite energy, comes intrinsically from the geometric properties of the boundary of sets of finite perimeter studied in the context of varifolds. In fact by De Giorgi's Theorem, if $E$ is a set of finite perimeter in $\R^2$ then the reduced boundary $\cF E$ is $1$-rectifiable, and therefore the integer rectifiable varifold $V_E=\bv(\cF E,1)$ is well defined. If a $1$-rectifiable varifold $V=\bv(\Ga,\te_V)$ has generalized curvature vector $k_V$, the analogue of the functionals \eqref{def1} and \eqref{def2} in the varifold context are defined by
\beq
	\cE_p(V)=\int_{\Ga} |k_V|^p \,d\mu_V,
\eeq
\beq
	\cF_p(V)= \mu_V(\R^2)+ \cE_p(V).
\eeq
So such elastic energies can be calculated on the varifold $V_E$ associated to a set of finite perimeter $E$, thus giving elasticity properties to the set $E$ in a pure intrinsic way.\\
We can introduce the class of elastic varifolds without boundary as the integer rectifiable varifolds $V=\bv(\Ga,\te_V)$ such that there exist a finite family of immersions $\ga_i:S^1\to\R^2$ such that
\beq \label{ref15}
	V=\sum_{i=1}^N (\ga_i)_\sharp(\bv(S^1,1)),
\eeq
where each $(\ga_i)\sharp(\bv(S^1,1))$ is the image varifold of $S^1$ induced by $\ga_i$. We shall see that a representation like \eqref{ref15} is not ambiguous and that the curves appearing in the formula can be used to compute the $\cF_p$ energy (Lemma \ref{lem2}).\\
In this way we will eventually define that a set $E$ is regular (in the sense that is has finite elastic energy) if
\beq \label{def3}
	\bv(\cF E,1)=\sum_{i=1}^N (\ga_i)_\sharp(\bv(S^1,1)),
\eeq
for some $C^2$ immersions $\ga_i:S^1\to\R^2$. In such a way the set in Fig. \ref{figotto} is considered to be regular, and it has finite elastic energy.\\

\noindent We have to mention that a significant attempt in order to give a good definition of the elastic energy on sets that are not natural limits of smooth sets with bounded energy is contained in \cite{BePa95} (see also the references therein). Here the authors consider an interesting generalization of the elastic energy functional whose relaxation is able to take into account the energy of angles and cusps.\\

\noindent Beside the study of the elastic properties of varifolds contained in Section \ref{secelastic}, there are other fundamental motivations for studying this alternative notion of relaxed energy. We would like to extend this ambient perspective and strategy (at least starting from the basic definitions) to the study of the relaxation of functionals depending on the curvature of surfaces in $\R^3$, such as the Willmore energy. Moreover this work is the starting point for the study of the gradient flow of the elastic energy of planar sets using an intrinsic definition of the functional, not completely relying on immersions covering the boundary of the set; the characterization of the relaxed energy allows us to define the gradient flow on a huge family of sets and therefore to try to obtain a generalized flow (for example using a minimizing movements technique in the spirit of \cite{LuSt95}, and this will be the reason of some assumptions we will make in the following). Observe that in particular in a generalized flow one certainly wants to consider sets like the one in Fig. \ref{figotto}, hence a definition in which its energy is finite is required (see also \cite{OkPoWh18}).\\

\noindent The paper is organized as follows. The first part of the work is devoted to the proof of some results about curves and varifolds with curvature from an ambient point of view. We prove a basic inequality concerning the elastic energy of immersed curves using a varifold perspective (Lemma \ref{lem0}), then we show an extension to $1$-dimensional varifolds with curvature in $\R^n$ together with uniform bounds on the multiplicity function (Theorem \ref{propmon}) and a monotonicity formula for the case $p=2$. This helps us to prove the main structural properties of elastic varifolds, which are contained in Lemma \ref{lem2} and Lemma \ref{lemV=V_E}. Such results are stated for any $p\in[1,\infty)$.\\
\noindent In the second part we focus on the $p>1$ case and we give a precise characterization of the $L^1$-relaxation of the energy $\cF_p$ starting from our new notion \eqref{def3} of regular set. The expression of the relaxed energy $\rfp(E)$ takes the form of a minimization problem defined on a class $\cA(E)$ of elastic varifolds suitably related to the set $E$ (Theorem \ref{thmmain}).\\
\noindent The relaxed energy $\rfp$ has to be compared with the classical results contained in \cite{BeDaPa93}, \cite{BeMu04} and \cite{BeMu07}, and in Subsection \ref{BM} we discuss an example of a set $E$ with finite relaxed energy $\rfp$ which is strictly less then its relaxed energy in the sense of \cite{BeMu04} (which is still finite however). The last part of the work continues with an application to a minimization problem arising from the inpainting problem in image processing (\cite{AmMa03}, \cite{BeCaMaSa11}). The relevance of our new definition of relaxed energy $\rfp$ is particularly evident in this application. Then we conclude the work with some comments on the qualitative properties of sets having finite relaxed energy; here we prove that a set $E$ with a boundary except that is smooth but at finitely many cusps has finite relaxed energy if and only if the number fo such cusps is even (Theorem \ref{thmcusps}), and we show that polygons always have infinite relaxed energy (Proposition \ref{poligoni}).\\

%%%%%%%%%%%%%%%%%%%%%%%%%%%%%%%%%%%%%%%%%%%%%%%%%%%%%%%%%%%%%%%%%%%%%%%%%%%
% FATTI PRELIMINARI

\section{Elastic energy of planar sets}\label{secelastic}

\subsection{Notation and definitions}

In the following if $\ga$ is any parametrization of a curve, we denote by $(\ga)$ its image. The letter $E$ will usually denote a measurable set in $\R^2$. We recall that $E$ has finite perimeter in an open set $\Om\con\R^2$ if the characteristic function $\chi_E$ restricted to $\Om$ belongs to $BV(\Om)$, and in such case we denote by $P(E,\Om)$ the perimeter of $E$ in $\Om$. For the theory of sets of finite perimeter we refer to \cite{AmFuPa00}.\\
\noindent If $E\con\R^2$ is measurable and $\Om\con\R^2$ is an open set such that $E$ has finite perimeter in $\Om$, we denote by $D\chi_E$ the gradient measure of $\chi_E$ and by $|D\chi_E|$ the corresponding total variation measure. Then we denote by
\beqs
\cF E=\bigg\{ x\in \supp |D\chi_E|\cap\Om\,\,\bigg|\,\, \exists\,\lim_{\ro\searrow0} \frac{D\chi_E(B_\ro(x))}{|D\chi_E(B_\ro(x))|}=:\nu_E(x),\,|\nu_E(x)|=1  \bigg\},
\eeqs
and we call $\cF E$ the reduced boundary of $E$, and $\nu_E$ is the generalized inner normal of $E$. By De Giorgi's Theorem the set $\cF E$ is $1$-rectifiable and $|D\chi_E|=\cH^1\res \cF E$.\\

\noindent Let $G(1,2)$ be the Cartesian product between $\R^2$ and the set of $1$-dimensional subspaces in $\R^2$. We call $G(1,2)$ the Grassmannian of $1$-dimensional spaces in $\R^2$. A point $(x,v)\in G(1,2)$ (where $v\in\R^2$ with $|v|=1$ generates the given $1$-dimensional subspace) is identified by the matrix $\pi_{x,v}$ that projects vectors in $T_x(\R^2)$ onto the subspace spanned by $v$; therefore $G(1,2)$ obtains a structure of metric space calculating the distance between two elements as the distance between the corresponding projection matrices. A $1$-dimensional varifold in $\R^2$ is a positive finite Radon measure on $G(1,2)$. For the theory of varifolds we refer to \cite{Si84}, and in this work we will always deal with integer rectifiable varifolds.\\
For a $1$-dimensional varifold $V$ in $\R^2$ we denote by $\mu_V$ the induced measure in $\R^2$. We recall that a $1$-dimensional rectifiable varifold $V=\bv(\Ga,\te_V)$ in $\R^2$ has generalized curvature vector $k_V\in L^1_{loc}(\mu_V)$ and generalized boundary $\si_V\in\cM^2(\R^2)$ if for any $X\in C^1_c(\R^2;\R^2)$ it holds
\beqs
\int \big( \div_{T_x\Ga} X \big)\,d\mu_V(x)= - \int \lgl X, k_V \rgl \, d\mu_V + \int X\,d\si_V.
\eeqs
Recall that is such case the measure $\si_V$ is singular with respect to $\mu_V$. If $f:\supp V\to \R^2$ is Lipschitz we define the image varifold $f_\sharp(V):=\bv(f(\Ga),\tilde{\te})$ with $\tilde{\te}(y)=\sum_{x\in f^{-1}(y)\cap \Ga} \te_V(x)$ for any $y\in\R^2$.\\
If $E$ has finite perimeter in $\R^2$ we denote by $V_E$ the associated varifold $V_E:=\bv(\cF E, 1)$. If a varifold $V=\bv(\Ga,\te_V)$ has generalized curvature $k_V$, then we define
\beqs
\cE_p(V):=\int |k_V|^p(x)\,d\mu_V(x),
\eeqs
while if $V$ does not admit generalized curvature we then set $\cE_p(V)=+\infty$.\\
At some point we will also use for a while some very basic facts about the theory of currents; for such definitions and results we refer to \cite{Si84}.\\

\noindent Here we recall with proof some basic properties of sets of finite perimeter, together with the choice of a convention and of the notation. The following observations actually work for sets of finite perimeter in any dimension.\\
If $E\con\R^2$ is a measurable set, for any $t\in[0,1]$ we denote by $E^t$ the subset of $t$-density points, that is
\beq
E^t:=\bigg\{ x\in \R^2 \,\, \bigg|\,\, \lim_{\ro \searrow 0 } \frac{|E\cap B_\ro(x)|}{|B_\ro(x)|}=t \bigg\}.
\eeq
The essential boundary $\pa^*E$ is then $\pa^*E:=\R^2 \sm (E^0\cup E^1)$. Recall that for a set of perimeter $E$ in $\R^3$ it holds that $\cF E\con E^{\frac{1}{2}}\con \pa^* E\con \pa E$ and $\cH^1(\pa^* E\sm \cF E)=0$.\\
\noindent Since in $BV$ we only consider equivalence classes of functions, it will be convenient to choose, for a finite perimeter set identified by a characteristic function $\chi_E\in BV(\R^2)$, the representative $E$ that is the set given by the points having density equal to $1$, i.e.
\begin{equation} \label{ax}
E=E^1
\coloneqq \bigg\{ x\in\R^2\,\,\bigg|\,\,\lim_{\rho\searrow 0} \frac{\int_{B_\rho(x)} \chi_E }{|B_\rho(x)|}=1\bigg\}.
\end{equation}
Assuming \eqref{ax}, we have that
\begin{equation}\label{eq:AxBordi}
\pa E \equiv \pa E^1 = \pa^m E \coloneqq \bigg\{ x\in\R^2 \,\,|\,\, \int_{B_\rho(x)} \chi_E>0,\, \int_{B_\rho(x)} 1-\chi_E>0 \,\, \forall \ro>0  \bigg\}.
\end{equation}
Indeed if $x\in \pa^m E$, as $E^1$ is a representative of $E$, there are sequences $x^1_n\in E^1,x^2_n\in \R^2\setminus E^1$ converging to $x$, and thus $x\in \pa E^1$. Conversely if $x\in\pa E^1$, then $\int_{B_\ro(x)} \chi_E >0$ for any $\ro>0$, for otherwise any point in $E^1$ sufficiently close to $x$ would have density equal to zero. Arguing similarly on the complement of $E^1$, \eqref{eq:AxBordi} follows.

\noindent It also follows that the reduced boundary is dense in the boundary of $E$, that is
\begin{equation} \label{ref2}
\overline{\cF E}=\pa^m E=\pa E.
\end{equation}
Indeed $\cF E\con \pa E$ and if by contradiction there is $x\in \pa^m E\sm \overline{\cF E}$, then for some $\ro_0>0$ we have $B_{\ro_0}(x)\cap \overline{\cF E}=\emptyset$ and $0< |E\cap B_{\ro_0}(x)|<\pi\ro_0^2$. Hence by relative isoperimetric inequality in the ball $B_{\ro_0}(x)$ we get that $ P(E, B_{\ro_0}(x) )>0$, but since $B_{\ro_0}(x)\cap \overline{\cF E}=\epty$ we also have $ P(E, B_{\ro_0}(x) ) = \cH^1(\cF E\cap  B_{\ro_0}(x))=0 $, which gives a contradiction.
Observe that it also follows that $\diam \cF E = \diam \pa E $.\\

\subsection{Preliminary estimates}

Here we prove a fundamental estimate concerning curves in $\R^2$.

\begin{lemma} \label{lem0}
	Let $\ga:\S^1\to \R^2$ be a regular curve in $W^{2,p}$ for some $p\in[1,\infty)$. Then
	\begin{equation} \label{ineq1}
	2\pi \le \int_{\S^1} |k_\ga|\,ds_\ga\le  \left(\int_{\S^1} |k_\ga|^p\,ds_\ga\right)^{\frac{1}{p}}(L(\ga))^{\frac{1}{p'}},
	\end{equation}
	where $L(\ga)$ denotes the length of the curve. Moreover, in the first inequality, equality holds if and only if $\ga$ parametrizes the boundary of a convex set and it is injective.
\end{lemma}

\begin{proof}
	By approximation it is enough to prove the statement for $\ga\in C^\infty$. Then the closed convex envelop of the support $(\ga)$ is a $C^{1,1}$-smooth set, and its boundary can be parametrized by an embedded curve $\si:\S^1\to \R^2$ of class $H^2$. Then $\int |k_\si|\,ds_\si \le \int |k_\ga|\,ds_\ga$. Hence, if we can prove that
	\begin{equation} \label{pluda}
	\int_{\S^1} |k_\si|\,ds_\si \ge 2\pi,
	\end{equation}
	the thesis follows by H\"{o}lder inequality. By approximation, we can assume without loss of generality that $\si:\S^1\to \R^2$ is an embedded curve of class $C^\infty$ that positively parametrizes the boundary of the bounded set it encloses. Moreover, as $\int |k_\si|\,ds_\si$ is scaling invariant, we can assume that $\si$ is parametrized by arclength. If $\tau_\si$ is the tangent vector along $\si$, we clearly have that $\tau_\si(\S^1)=\S^1$, that is, it is surjective. Indeed, $\si$ has index equal to $1$, hence the degree of $\tau_\si$ equals $1$, and then $\tau_\si:\S^1\to\S^1$ has to be surjective. Therefore we estimate
	\[
	\int_{\S^1} |k_\si|\,ds_\si = \int_{\S^1} |\pa_x \tau_\si(x)|\,dx = \int_{\S^1} \sharp \tau_\si^{-1}(v)\,d\cH^1(v) \ge \cH^1(\S^1)=2\pi,
	\]
	where we used the area formula.
	Letting $\tau_\si(x)=(\cos\te(x),\sin\te(x))$ for a smooth angle function $\te:[0,2\pi]\to\R$, we have that $|k_\si|=|\pa_x\te|$. Hence the equality case in the above estimate implies that $\pa_x\te$ has a sign. As $\pa_x\te=\langle k_\si,\nu_\si\rangle$, this means that $\si$ is convex.
	%	
	%	
	%	Let $\si:[0,L]\to\R^2$ the arclength parametrization of the curve. We want to prove that
	%	\begin{equation}
	%	\int_0^L |k_\si|\ge 2\pi,
	%	\end{equation}
	%	so that applying Holder inequality on the left, \eqref{pluda} will give the claim. Let $T:[0,L]\to \S^1$ be the unit tangent vector $T=\dot \si$. The support $(\si)$ is compact, then there are two parallel lines $\{x=\al \},\{x=\be \}$ such that $(\si)\con\{ \al\le x\le \be \}$ and $(\si)$ touches tangentially the two lines at points $\si(0)=x_0,\si(b)=x_b$ (up to reparametrization). Assume that $T(0)=(0,1)$. By continuity there exists $a\in(0,b)$ such that $T(a)=(-1,0)$ and $c\in(b,L)$ such that $T(c)=(1,0)$. Now two cases can occur: $T(b)=(0,1)$ or $T(b)=(0,-1)$. If $T(b)=(0,-1)$, then $\S^1\con T([0,L])$; if $T(b)=(0,1)$, then $\S^1\cap\{(x,y)\,|\,y\ge0\} \con T([0,L])$ and for any $v\in \S^1\cap\{(x,y)\,|\,y>0\}$ we have that $\sharp T^{-1}(v)\ge2 $. Therefore using the area formula (\Cref{thm:AreaFormula}) we obtain
	%	\begin{equation*}
	%	\begin{split}
	%	\int_0^L |k_\si|(s)\,ds = \int_0^L  |T'|(s)\,ds =\int_{\S^1} \sharp T^{-1}(v)\,d\cH^1(v) \ge 2\pi.
	%	\end{split}
	%	\end{equation*}
	%	Writing $T(s)=(\cos\te(s),\sin\te(s))$ for an angle function $\te(s)$, we have that $|T'|=|\te'|=|k_\si|$, and then equality hold if and only if $\te'$ has a sign and the curve has index equal to $1$. Hence this completes the proof.
\end{proof}

\noindent We mention that inequality \eqref{pluda} is already present in \cite{DaNoPl18}, proved with a different method in the setting of networks, but in the following we will need the specific approach used in the proof of Lemma \ref{lem0}. Also, we are going to prove that inequality \eqref{ineq1} is true (up to changing the constant) in an analogous sense in the setting of varifolds as stated in the inequality \eqref{ineq1var} in Subsection \ref{sscmon}.\\

\subsection{Monotonicity} \label{sscmon}

Here we develop a monotonicity-type argument that is the direct analogue of Simon's Monotonicity Formula (\cite{Si93}), which is fundamental in the study of the Willmore energy, that in some sense is the two-dimensional energy corresponding to the functional $\cE_2$. This result is of independent interest and it will be sated in general in $\R^n$.\\

\noindent Throughout this Subsection consider $x_0\in \R^n$, $0<\si<\ro<+\infty$, $V=\bv(\Ga,\te_V)\neq0$ an integer $1$-dimensional rectifiable varifold in $\R^n$ with curvature $k_V$ such that $\cE_p(V)<+\infty$ for some $p\in[1,\infty)$ (for the moment $\mu_V$ is just locally finite on $\R^n$). Also we are assuming that $\si_V=0$.\\

\noindent Consider the field $X(x)=\big(\frac{1}{|x-x_0|_\si}-\frac{1}{\ro}\big)_+ x$, where $(\cdot)_+$ denotes the nonnegative part and $|\cdot|_\si=\max\{|\cdot|,\si\}$. For any set $E$ let $E_\si=E\cap B_\si$, $E_\ro=E\cap B_\ro$, $E_{\ro,\si}=E_\ro \sm \overline{B_\si}$. Then the tangential divergence of $X$ is
\beqs
\div_{T\Ga} X (x)= \begin{cases}
	\frac{1}{\si}-\frac{1}{\ro} & \mbox{on }\Ga_\si,\\
	\frac{|(x-x_0)^\perp|^2}{|x-x_0|^3}-\frac{1}{\ro} & \mbox{on }\Ga_{\ro,\si}.
\end{cases}
\eeqs

\noindent We want to prove the following result.
\begin{thm} \label{propmon}
	Under the above assumptions it holds that
	\beq \label{densbound'}
	\limsup_{\si\searrow0} \frac{\mu_V(B_\si(x_0))}{\si} \le \liminf_{\ro\nearrow\infty} \frac{\mu_V(B_\ro(x_0))}{\ro}+\int_{B_\ro(x_0)} |k_V|\,d\mu_V(x).
	\eeq
	If also $\mu_V(\R^n)<+\infty$, then
	\beq \label{ineq1var}
	2\le \cE_1(V)\le \mu_V(\R^n)^{\frac{1}{p'}}\cE_p(V)^{\frac{1}{p}},
	\eeq
	and we have the following bounds on the multiplicity function:
	\beq
		p>1 \quad \Rightarrow \quad \te_V(x) \le  \frac{1}{2}\cE_1(V) \qquad\forall\,x\in \R^n,
	\eeq
	\beq
		p=1 \quad \Rightarrow \quad \te_V(x)\le  \frac{1}{2}\cE_1(V) \qquad\mbox{for }\cH^1\mbox{-ae}\,\,x\in \R^n.
	\eeq
	If also $\mu_V(\R^n)<+\infty$ and $\Ga$ is essentially bounded, i.e. $\cH^1(\Ga\sm B_R(0))=0$ for $R$ large enough, then
	\beq
		p=1 \quad \Rightarrow \quad \exists \lim_{r\searrow0}\frac{\mu_V(B_r(x))}{2r}=\te_V(x) \le  \frac{1}{2}\cE_1(V) \qquad\forall\,x\in \R^n.
	\eeq
\end{thm}

\begin{proof}
Integrating the divergence $\div_{T\Ga} X $ above with respect to $\mu_V$ and using the first variation formula we get
\beq \label{ref3}
\begin{split}
	&\frac{\mu_V(B_\si(x_0))}{\si}+\frac{1}{\si}\int_{B_\si(x_0)} \lgl k_V, x-x_0\rgl\,d\mu_V(x) + \int_{B_\ro(x_0)\sm B_\si(x_0)} \frac{|(x-x_0)^\perp|^2}{|x-x_0|^3} \,d\mu_V(x) = \\ &\,\,\, = \frac{\mu_V(B_\ro(x_0))}{\ro}+\frac{1}{\ro}\int_{B_\ro(x_0)} \lgl k_V, x-x_0\rgl\,d\mu_V(x) - \int_{B_\ro(x_0)\sm B_\si(x_0)} \bigg\lgl k_V, \frac{x-x_0}{|x-x_0|}\bigg\rgl\,d\mu_V(x).
\end{split}
\eeq
Dropping the positive term on the left we obtain
\beqs
\begin{split}
	\frac{\mu_V(B_\si(x_0))}{\si} &+\frac{1}{\si}\int_{B_\si(x_0)} \lgl k_V, x-x_0\rgl\,d\mu_V(x)\le \\ &\le \frac{\mu_V(B_\ro(x_0))}{\ro}+ \int_{B_\ro(x_0)} \bigg\lgl k_V, \frac{x-x_0}{\ro} - \frac{x-x_0}{|x-x_0|}\chi_{B_\ro(x_0)\sm B_\si(x_0)} \bigg\rgl \, d\mu_V(x).
%	
%	
%	
%	 \|k_V\|_{L^p(B_\ro(x_0),\mu_V)}\mu_V(B_\ro(x_0))^{\frac{1}{p'}} + \\ &\qquad\qquad\qquad\qquad\qquad\qquad\quad+ \|k_V\|_{L^p(B_\ro(x_0)\sm B_\si(x_0),\mu_V)}\mu_V(B_\ro(x_0)\sm B_\si(x_0))^{\frac{1}{p'}}  \le\\ & \le  \frac{\mu_V(B_\ro(x_0))}{\ro}+ 2 \mu_V(B_\ro(x_0))^{\frac{1}{p'}} \|k_V\|_{L^p(B_\ro(x_0),\mu_V)}.
\end{split}
\eeqs
Since
\beqs
	\bigg|\frac{1}{\si}\int_{B_\si(x_0)} \lgl k_V, x-x_0\rgl\,d\mu_V(x)\bigg|\le \bigg(\int_{B_\si(x_0)}|k_V|^p\,d\mu_V\bigg)^{\frac{1}{p}}\big(\mu_V(B_\si(x_0)\big)^{\frac{1}{p'}}\xrightarrow[\si\to0]{}0,
\eeqs
and
\beqs
	\frac{x-x_0}{|x-x_0|}\chi_{B_\ro(x_0)\sm B_\si(x_0)} \xrightarrow[\si\to0]{} \frac{x-x_0}{|x-x_0|}\chi_{B_\ro(x_0)} \qquad \mbox{in }L^{p'}(\mu_V),
\eeqs
letting $\si\searrow0$ and then $\ro\nearrow\infty$ we get the inequality
\beqs
\begin{split}
	\limsup_{\si\searrow0} \frac{\mu_V(B_\si(x_0))}{\si} &\le \liminf_{\ro\nearrow\infty} \frac{\mu_V(B_\ro(x_0))}{\ro}+ \int_{B_\ro(x_0)} \bigg\lgl k_V, \bigg(\frac{1}{\ro}-\frac{1}{|x-x_0|}\bigg)(x-x_0)\bigg\rgl \, d\mu_V(x)\le\\
	&\le \liminf_{\ro\nearrow\infty} \frac{\mu_V(B_\ro(x_0))}{\ro}+\int_{B_\ro(x_0)} |k_V| \bigg|\frac{|x-x_0|}{\ro}-1 \bigg|\,d\mu_V(x)\le\\
	&\le \liminf_{\ro\nearrow\infty} \frac{\mu_V(B_\ro(x_0))}{\ro}+\int_{B_\ro(x_0)} |k_V|\,d\mu_V(x).
\end{split}
\eeqs
that is \eqref{densbound'}.

\noindent Suppose from now on that $\mu_V(\R^n)<+\infty$, then \eqref{densbound'} gives
\beq \label{densbound}
\limsup_{\si\searrow0} \frac{\mu_V(B_\si(x_0))}{\si} \le   \cE_1(V).
\eeq
Equation \eqref{densbound} gives us the pointwise bounds on the multiplicity function $\te_V$ as follows.\\
If $p>1$ we know that the density $\lim_{\si \searrow 0 } \frac{\mu_V(B_\si(p))}{2\si}$ exists at any $p$ and can be used as multiplicity function $\te_V$ for $V$ (\cite{Si84}, page 86). So in this case \eqref{densbound} gives
\beq
\te_V(x)= \lim_{\si \searrow 0 } \frac{\mu_V(B_\si(p))}{2\si}  \le \frac{1}{2} \cE_1(V)\qquad\forall\,x\in \Ga.
\eeq\\
Instead in the $p=1$ case we can say the following. Since $\Ga$ has generalized tangent space at $\cH^1$-ae point we have that
\beq
	\te_V(x)=\lim_{\si\searrow0}  \frac{\mu_V(B_\si(x_0))}{2\si} \le  \frac{1}{2}\cE_1(V) \qquad\mbox{for }\cH^1\mbox{-ae}\,\,x\in \Ga.
\eeq
Therefore, since $\te_V(x)\ge1$ at some point $x$, for any $p\in[1,\infty)$ we can state inequality \eqref{ineq1var}.\\

\noindent Now assume $p=1$ and without loss of generality $\Ga\con B_{R_0}(0)$ is bounded, then we want to show that the limit $\lim_{\si \searrow 0 } \frac{\mu_V(B_\si(x_0))}{\si}$ does exist for any $x_0\in\R^n$. In fact in Equation \eqref{ref3} we have
\beqs
	\begin{split}
				& \bigg|\frac{1}{\si}\int_{B_\si(x_0)} \lgl k_V, x-x_0\rgl \, d\mu_V(x)\bigg| \to 0 \qquad \mbox{as }\si\to0,\\
				& \bigg|\frac{1}{\ro}\int_{B_\ro(x_0)} \lgl k_V, x-x_0\rgl \, d\mu_V(x)\bigg| \le \frac{R_0}{\ro}\int_{B_{R_0}(x_0)} |k_V|\,d\mu_V \to 0 \qquad \mbox{as }\ro\to\infty,\\
				& \frac{x-x_0}{|x-x_0|} \chi_{B_\ro(x_0)\sm B_\si(x_0)} \to \frac{x-x_0}{|x-x_0|} \qquad\mbox{in }L^1(\R^n,\mu_V),
	\end{split}
\eeqs\\
where the last statement follows by Dominated Convergence. Therefore there exists the limit
\beqs
\lim_{\si \searrow 0,\,\, \ro \nearrow\infty }\int \bigg\lgl k_V, \frac{x-x_0}{|x-x_0|} \chi_{B_\ro(x_0)\sm B_\si(x_0)}\bigg\rgl d\mu_V(x),
\eeqs
which is also finite. Hence \eqref{ref3} implies that
\beq
	\sup_{\si,\ro>0} \int_{B_\ro(x_0)\sm B_\si(x_0)} \frac{|(x-x_0)^\perp|^2}{|x-x_0|^3} \,d\mu_V(x) <+\infty,
\eeq
thus by monotonicity the limit
\beqs
	\lim_{\si \searrow 0,\,\, \ro \nearrow\infty } \int_{B_\ro(x_0)\sm B_\si(x_0)} \frac{|(x-x_0)^\perp|^2}{|x-x_0|^3} \,d\mu_V(x)
\eeqs
exists finite. Since $\lim_{\ro\to\infty} \frac{\mu_V(B_\ro(x_0))}{\ro}\to 0$ as $\ro\to 0$, Equation \eqref{ref3} implies that
\beq
	\exists \lim_{\si \searrow 0 } \frac{\mu_V(B_\si(x_0))}{\si} <+\infty \qquad \forall x_0\in \R^n,
\eeq
which completes the proof.
\end{proof}

\noindent We mention that the inequality \eqref{ineq1var} is probably not sharp, but still new in the context of $1$-dimensional varifolds.\\

\noindent We conclude with a monotonicity statement concerning the $p=2$ case.
\begin{remark}
Let $p=2$. For $r>0$ let
\beqs
	A(r)=\bigg(\frac{1}{2}+\frac{1}{r}\bigg)\mu_V(B_r(x_0))+ \frac{1}{r}\int_{B_r(x_0)} \lgl k_V, x-x_0 \rgl \, d\mu_V(x) + \frac{1}{2}\int_{B_r(x_0)}|k_V|^2\,d\mu_V.
\eeqs
Then
\beq \label{monotonicity}
	A(\si)+ \int_{B_\ro(x_0)\sm B_\si(x_0)} \bigg( \frac{|(x-x_0)^\perp|^2}{|x-x_0|^3}+ \frac{1}{2}\bigg| k_V + \frac{x-x_0}{|x-x_0|} \bigg|^2\bigg) \, d\mu_V(x) = A(\ro),
\eeq
in particular $r\mapsto A(r)$ is nondecreasing.\\
Indeed to prove \eqref{monotonicity} just insert the identity $ \big\lgl k_V, \frac{x}{|x|}\big\rgl =\frac{1}{2}\big( \big|k_V +  \frac{x}{|x|}\big|^2 - |k_V|^2  - 1 \big) $ in \eqref{ref3}.\\
\noindent Moreover if we additionally require that $\mu_V(\R^n)<+\infty$, then
\beqs
\begin{split}
  	\frac{1}{R}&\int_{B_R(x_0)} \lgl k_V, x-x_0 \rgl \, d\mu_V(x)= \frac{1}{R}\int_{B_{r}(x_0)} \lgl k_V, x-x_0 \rgl \, d\mu_V(x) + \frac{1}{R}\int_{B_R(x_0)\sm B_r(x_0)} \lgl k_V, x-x_0 \rgl \, d\mu_V(x)\le \\
  	& \le \frac{1}{R}\int_{B_{r}(x_0)} \lgl k_V, x-x_0 \rgl \, d\mu_V(x)  + \bigg(\int_{B_R(x_0)\sm B_r(x_0)} |k_V|^2 \bigg)^{\frac{1}{2}} \big(\mu_V(B_R(x_0)\sm B_r(x_0)) \big)^{\frac{1}{2}}
\end{split}
\eeqs
for any $r<R$. So letting first $R\to\infty$ and then $r\to\infty$ we get that $\frac{1}{R}\int_{B_R(x_0)} \lgl k_V, x-x_0 \rgl \, d\mu_V(x) \to 0$ as $R\to\infty$. And thus we obtain that
\beq
	\lim_{r\to\infty} A(r)= \frac{1}{2}\bigg(\mu_V(\R^n) + \cE_2(V) \bigg),
\eeq
for any choice of $x_0\in\R^n$.
\end{remark}

\begin{remark}
	After the conclusion of the work the author became aware of the fact that Theorem \ref{propmon} also follows from Corollary 4.8 in \cite{Me16} (see also Theorem 3.5 in \cite{MeSc18}).\\
\end{remark}

\subsection{Elastic varifolds}

\noindent Here we prove some important remarks about varifolds defined through immersions of elastic curves. The next definition comes from \cite{BeMu04}.

\begin{defn}
	Given a family of regular $C^1$ curves $\al_i:(-a_i,a_i)\to\R^2$ for $i=1,...,N$ and a point $p\in \R^2$ such that $\al_i(t_i)=p$ for some times $t_i$ and the curves $\{\al_i\}$ are tangent at $p$. Let $v\in S^1$ such that $\al_i'(t_i)$ and $v$ are parallel for any $i$. We say that $R_v(p)$ is a \emph{nice rectangle at} $p$ \emph{for the curves} $\{\al_i\}$ \emph{with side parameters} $a,b>0$ if
	\beqs
	R_v(p)=\{ z\in\R^2\,\,:\,\, |\lgl z-p,v\rgl|<a,\, |\lgl z-p,v^\perp\rgl|<b \},
	\eeqs
	and
	\beqs
	R_v(p)\cap \bigg(\bcup_{i=1}^N (\al_i) \bigg) = \bcup_{i=1}^M graph(f_i),
	\eeqs
	for distinct $C^1$ functions $f_i:[-a,a]\to(-b,b)$, where $graph(f_i)$ denotes the graph of $f_i$ constructed on the lower side of the rectangle.\\
\end{defn}

\noindent We also give the following definition.
\begin{defn}\label{defflux}
	Let $V=\mathbf{v}(\cup_{i\in I}(\gamma_i),\theta_V)$ be a varifold defined by the $C^1\cap W^{2,p}$ immersions $\gamma_i:S^1\to \mathbb{R}^2 $, and assume that $\mathcal{F}_p(V)<+\infty,\,\theta_V\le C<+\infty$.\\
	For any $ p\in \cup_{i\in I}(\gamma_i)$ and any $v\in S^1$ denote by $g_1,...,g_r:[-\varepsilon,\varepsilon]\hookrightarrow\mathbb{R}^2$ arclength parametrized injective arcs such that: $g_i(0)=p$, $\dot{g}_i(0) = v$, $g_i([-\varepsilon,0])\neq g_j([-\varepsilon,0])$ or $g_i([0,\varepsilon])\neq g_j([0,\varepsilon])$ for $i\neq j$, and $\cup_{i=1}^r (g_i) \cap \overline{B}_\rho(p) = \cup_{i\in I}(\gamma_i)\cap \overline{B_\rho(p)}$. Observe that for any such $p,v$ and $\rho$ small enough, the arcs $g_i$ are well defined.\\
	We say that $V$ verifies the \emph{flux property} if: $\forall\, p\in \cup_{i\in I}(\gamma_i),\,\forall \, v\in S^1,$ and $\rho$ small enough there exists a nice rectangle $R_v(p)\subset B_\rho(p)$ for the family of arcs $\{g_i\}$ such that it holds that
	\begin{equation*}
	\forall |c|<a:\qquad \sum_{ z\in \cup_{i=1}^r (g_i) \cap \{y\,|\,\langle y-p, v\rangle = c \} }  \theta_V(z)=M,
	\end{equation*}
	for a constant $M\in\mathbb{N}$ with $M\le \theta_V(p)$.\\
\end{defn}
\noindent Roughly speaking, Definition \ref{defflux} requires that the ``incoming" total amount of multiplicity at $p$ in direction $v$ equals the ``outcoming" total amount of multiplicity at $p$ in direction $v$.\\
\noindent Observe that if $V=\sum_{i\in I} (\ga_i)_\sharp(\bv(S^1,1))$ with $\ga_i\in C^1\cap W^{2,p}$ immersions and $\fp(V)<+\infty,\,\te_V\le C<+\infty$, then $V$ verifies the flux property.\\

\begin{remark} \label{rem1}
	Let $E$ be a set of finite perimeter in $\R^2$, let $\Ga=\cup_{i=1}^N (\ga_i)$ with $\ga_i\in C^1(S^1;\R^2)$ and regular for any $i$. Assume that $V_E:=\bv(\cF E, 1) = \sum_{i=1}^N (\ga_i)_\sharp(\bv(S^1,1))$. Then $\cH^1(\pa E\sm \cF E)=0$, and we can equivalently write $V_E=\bv(\pa E, 1)$.\\
	In fact by assumption $\cH^1$-almost every point $p\in \Ga$ is contained in $\cF E$, $\supp V_E=\Ga$, and $\Ga=\supp V_E=\supp (\cH^1\res\cF E) = \pa E$. Therefore $0=\cH^1(\Ga\sm\cF E)=\cH^1(\pa E\sm \cF E)$.\\
\end{remark}

\begin{lemma} \label{lem2}
	Assume $p>1$. If an integer rectifiable varifold $V=\bv(\Ga,\te_V)$ is such that $V=\sum_{i=1}^N(\ga_i)_\sharp(\bv(S^1,1))$ for some regular curves $\ga_i \in W^{2,p}(S^1;\R^2)$ and $\fp(V)<+\infty$, then $V$ has generalized curvature
	\beq \label{ref7'}
	k_V(p)=\frac{1}{\te_V(p)}\sum_{i=1}^N \sum_{t\in\ga_i^{-1}(p)} k_{\ga_i}(t)  \qquad\mbox{at }\cH^1\mbox{-ae }p\in \Ga,
	\eeq
	the generalized boundary $\si_V=0$, and
	\beq \label{ref7}
	\cE_p(V)= \sum_{i=1}^N\cE_p(\ga_i).
	\eeq
	In particular, since $k_V$ is unique, the value $\cE_p(V)$ is independent of the choice of the family of curves $\{\ga_i\}$ defining $V$.
\end{lemma}

\begin{proof}
	In fact suppose first that $N=1$, and then call $\ga_1=\ga$. Up to rescaling, assume without loss of generality that $\ga$ is an arclength parametrization. By assumption $\ga\in C^{1,\al}$ for $\al\le\frac{1}{p'}$, and clearly $\Ga=(\ga)$ and $\mu_V=\te_V \cH^1\res(\ga)=\ga_\sharp(\cH^1\res S^1)$ (by the arclength parametrization assumption). %<- scritto a pagina 3 dei fatterelli
	If $X\in C^1_c(\R^2;\R^2)$ is a vector field, using the area formula and the fact that $\te_V\ge 1$ $\cH^1$-ae on $\Ga$, we have
	\beqs
	\begin{split}
		\int \div_{T_p\Ga} X \, d\mu_V(p) 
		&= \int \div_{T_p(\ga)} X \, d\ga_\sharp(\cH^1\res S^1)(p)= \int_{S^1} \lgl \ga'(t), (\nabla X)_{\ga(t)}(\ga'(t)) \rgl \, dt =\\ &= \int_{S^1} \lgl \ga', (X\circ \ga)'\rgl\,dt=  - \int_{S^1} \lgl \ga''(t), X(\ga(t)) \rgl \, dt = \\
		& = - \int \int_{\ga^{-1}(p)} \lgl \ga''(t),X(\ga(t))\rgl \,d\cH^0\,d\cH^1\res\Ga(p) =\\
		&= -\int \bigg\lgl X(p), \int_{\ga^{-1}(p)} k_\ga(t) \,d\cH^0 \bigg\rgl \frac{\te_V(p)}{\te_V(p)} \, d\cH^1\res \Ga (p)=\\
		& = - \int \bigg\lgl X(p), \frac{1}{\te_V(p)} \sum_{t\in\ga^{-1}(p)} k_\ga(t) \bigg\rgl \, d\mu_V(p).
	\end{split}
	\eeqs
	If now $N>1$, by linearity of the first variation we get
	\beqs
	\begin{split}
		\int \div_{T_p\Ga} X \, d\mu_V(p) 
		&=- \sum_{i=1}^N \int \bigg\lgl X(p),  \sum_{t\in\ga_i^{-1}(p)} k_{\ga_i}(t) \bigg\rgl \, d\cH^1\res(\ga_i)(p) =\\
		&= - \int \bigg\lgl  X(p), \frac{1}{\te_V(p)}\sum_{i=1}^N \sum_{t\in\ga_i^{-1}(p)} k_{\ga_i}(t) \bigg\rgl \te_V(p) \,d\cH^1\res\big(\cup_{i=1}^n(\ga_i)\big) =\\
		&=- \int \bigg\lgl  X(p), \frac{1}{\te_V(p)}\sum_{i=1}^N \sum_{t\in\ga_i^{-1}(p)} k_{\ga_i}(t) \bigg\rgl \,d\mu_V.
	\end{split}
	\eeqs
	Now we want to prove \eqref{ref7}. Let us consider the set $W=\{ p\in\Ga\,|\,\te_V(p)>1 \}$. Up to redefining some $\ga_i$ on another circumference, we can suppose from now on that $\ga_i$ is an arclength parametrization. We can write $W=T\cup X\cup Y\cup Z$, with
	\beqs
	\begin{split}
		&T=\big\{ p\in W\,|\, \exists i,j,t,\tau:\,\,\ga_i(t)=\ga_j(\tau)=p,\,\ga_i'(t)\neq \al\ga_j'(\tau) \,\,\forall  \al\in\R   \big\},\\
		&X=\big\{ p\in W\sm T\,|\, \exists i,t:\,\,\ga_i(t)=p,\,t \mbox{ is not a Lebesgue point of }\ga_i'' \big\},\\
		&Y=\big\{ p\in W\sm (T\cup X)\,|\,    	\forall i,j,t,\tau:\,\,\ga_i(t)=\ga_j(\tau)=p \,\Rightarrow\, \ga_i''(t)=\ga_j''(\tau) \big\},\\
		&Z=\big\{ p\in W\sm (T\cup X)\,|\,    	\exists i,j,t,\tau:\,\,\ga_i(t)=\ga_j(\tau)=p ,\, \ga_i''(t)\neq\ga_j''(\tau) \big\}.
	\end{split}
	\eeqs
	We are going to prove that $T,Z$ are at most countable, then since $\cH^1(X)=0$ we will get that $\cH^1(W)=\cH^1(Y)$. Hence by \eqref{ref7'} one immediately gets \eqref{ref7}.\\
	Let $p\in \Ga$ and $C\in \N$ such that $\te_V\le C$ . Let $v_1(p),...,v_k(p)\in S^1$ with $k=k(p)\le C$ such that if $\ga_i(t)=p$ then $\ga_i'(t)$ is proportional to some $v_j$. For any $i=1,...,k$ let $R_{v_i}(p)$ be a nice rectangle at $p$ for the curves $\{\al_j\}_{j\in J(i)}$ which are suitable restrictions of the curves $\{\ga_i\}$. Then let $f^i_1,...,f^i_l$ with $l=l(i)$ be $C^1$ functions $f_s^i:[-a_i,a_i]\to(-b_i,b_i)$ given by the definition of nice rectangle.\\
	Let $q\in \cup_{s=1}^l graph(f^i_s)$, and assume $q\in T$. If $a_i$ is chosen sufficiently small, the fact that $q$ belongs to $T$ means that the transversal intersection happens between some of the curves $\{\al_j\}_{j\in J(i)}$. This means that there is some $\de_q>0,x_q\in (-a_i,a_i),r,s\in\{1,...,l\}$ such that
	\beqs
		f^i_r(x_q)=f^i_s(x_q),\quad(x_q,f^i_r(x_q))=q,\quad graph\big(f^i_r|_{(x_q-\de_q,x_q+\de_q)}\big) \cap graph\big(f^i_s|_{(x_q-\de_q,x_q+\de_q)}\big)=\{q\}.
	\eeqs
	Letting $A_i=\{x\in(-a_i,a_i)\,|\,f^i_r\neq f^i_s\}$, which is open, we see that $x_q$ belongs to the boundary of some connected component of $A_i$. This implies that $T\cap \big(\cup_{s=1}^l graph(f^i_s)\big)$ is countable, and this is true for any $i=1,...,k(p)$.\\
	For any $p\in\Ga$ take a ball $B_{r(p)}(p)\con \cap_{i=1}^{k(p)} R_{v_i(p)}(p)$ for suitable rectangles $R_{v_i(p)}(p)$ as above. Then $T\cap B_{r(p)}(p)$ is countable. Since $\Ga$ is compact, taking a finite cover of such balls $B_{r(p_1)}(p_1),...,$ $B_{r(p_L)}(p_L)$, we conclude that $T$ is countable.\\
	
	\noindent Consider now $q\in \cup_{s=1}^l graph(f^i_s)$, and assume $q\in Z$. If $a_i$ is chosen sufficiently small, the fact that $q$ belongs to $Z$ means that the tangential intersection happens between some of the curves $\{\al_j\}_{j\in J(i)}$. Hence at some $x_q\in (-a_i,a_i)$ for some $r,s\in \{1,...,l\}$ we find that $x_q$ is a Lebesgue point for $(f^i_r)''$ and $(f^i_s)''$, and
	\beqs
		f^i_r(x_q)=f^i_s(x_q),\quad(x_q,f^i_r(x_q))=q,\quad (f^i_r)''(x_q)\neq (f^i_s)''(x_q).
	\eeqs
	This implies that there exists $\ep>0$ such that for any $0<|t|<\ep$ we have $(f^i_r)'(x_q+t)\neq (f^i_s)'(x_q+t)$. By continuity of the first derivative we have that, for example, $(f^i_r)'(x_q+t)>(f^i_s)'(x_q+t)$ for any $0<|t|<\ep$, and therefore $ f^i_r(x_q+t)>f^i_s(x_q+t)$ for any $0<|t|<\ep$. So we find that $x_q$ belongs to the boundary of a connected component of an $A_i$ defined as above as in the case of the set $T$. Arguing as before we eventually get that $Z$ is countable.\\
\end{proof}

\begin{lemma} \label{lemV=V_E}
	Let $\ga_1,...,\ga_N:S^1\to\R^2$ be Lipschitz curves and let $V=\bv(\Ga,\te_V)=\sum_{i=1}^N (\ga_i)_\sharp(\bv(S^1,1))$. Assume that $\cH^1(\{x\,|\,\te_V(x)>1 \})=0$, and define
	\beq \label{ref9}
		E=\bigg\{ p\in\R^2\sm\Ga\,:\, \bigg|\sum_{i=1}^N \ind_{\ga_i}(p)\bigg| \mbox{ is odd}  \bigg\}.
	\eeq
	Then $V=V_E:=\bv(\cF E,1)$, and $E$ is uniquely determined by $V$, i.e. if $V=\sum_{i=1}^N (\ga_i)_\sharp(\bv(S^1,1))=\sum_{i=1}^M (\si_i)_\sharp(\bv(S^1,1))$ then the corresponding set $E$ defined using \eqref{ref9} with the family $\{\ga_i\}$ is the same set defined using \eqref{ref9} with the family $\{\si_i\}$.
\end{lemma}

\begin{proof}
	The set $E$ is open and bounded and $\pa E=\Ga$, hence $E$ is a set of finite perimeter.
	Let us first check that if $V=\sum_{i=1}^N (\ga_i)_\sharp(\bv(S^1,1))=\sum_{i=1}^M (\si_i)_\sharp(\bv(S^1,1))$, then the definition of $E$ by \eqref{ref9} is independent of the choice of the family of curves. The fact that a point $p\in\R^2\sm\Ga$ belongs to $E$ depends on the residue class
	\beqs
		\bigg(\sum_{i=1}^N \ind_{\ga_i}(p)\bigg) \quad\mbox{mod }2, \qquad\qquad \mbox{or }\quad \bigg(\sum_{i=1}^M \ind_{\si_i}(p)\bigg) \quad\mbox{mod }2.
	\eeqs
	Without loss of generality we think that $p=0$. Since the curves $\{\ga_i\},\{\si_i\}$ define the same varifold, for $\cH^1$-ae point $q\in \{(x,y)\in\R^2\,|\,x^2+y^2=1\}$ we have that
	\beq \label{ref18}
		\sum_{i=1}^N\sharp\bigg(\frac{\ga_i}{|\ga_i|}\bigg)^{-1}(q)=\sum_{i=1}^M\sharp\bigg(\frac{\si_i}{|\si_i|}\bigg)^{-1}(q).
	\eeq
	In the following we denote by $deg(f,y)$ the degree of a map $f$ at $y$ and by $deg_2(f,y)$ the degree mod $2$ of $f$ at $y$ (we refer to \cite{Mi65}). Since the curves are Lipschitz almost every point $q\in \{(x,y)\in\R^2\,|\,x^2+y^2=1\}$ is a regular value for $\frac{\ga_i}{|\ga_i|},\frac{\si_i}{|\si_i|}$ and we can perform the calculation
	\beqs
		\begin{split}
			\bigg(\sum_{i=1}^N \ind_{\ga_i}(p)\bigg) \quad\mbox{mod }2 &
			 = \bigg(\sum_{i=1}^N deg\bigg(\frac{\ga_i}{|\ga_i|},q\bigg)\bigg) \quad\mbox{mod }2 = \bigg(\sum_{i=1}^N deg_2\bigg(\frac{\ga_i}{|\ga_i|},q\bigg)\bigg) \quad\mbox{mod }2=\\
			 &= \bigg(\sum_{i=1}^N\sharp\bigg(\frac{\ga_i}{|\ga_i|}\bigg)^{-1}(q) \quad\mbox{mod }2\bigg) \qquad \mbox{mod }2,
		\end{split}
	\eeqs
	that together with the same expression using the curves $\si_i$, implies that $E$ is uniquely defined by \eqref{ref18}.\\
%	If we prove that $V=V_E$, then we get that $E$ is uniquely defined by $V$. In fact suppose that $V=\sum_{i=1}^N (\ga_i)_\sharp(\bv(S^1,1))=\sum_{i=1}^M (\si_i)_\sharp(\bv(S^1,1))$, and let $E_\ga,E_\si$ be the two sets defined using \eqref{ref9} respectively with the family $\{\ga_i\}$ or the family $\{\si_i\}$, so that $V=V_{E_\si}=V_{E_\ga}$. In particular $\cH^1\res\cF E_\si=\cH^1\res \cF E_\ga$. Suppose that $E_\si\neq E_\ga$, and so that there exists $x\in E_\si\sm E_\ga$. Then by \eqref{ax} and \eqref{ref4} there exists $\ro_0>0$ such that $B_{\ro_0}(x)\cap E_\ga=\epty$ and $|B_\ro(x)\cap E_\si|>0$ for any $0<\ro<\ro_0$. Hence $x\not\in\Ga$, and thus $x\in E^i_\si$ for some connected component $E^i_\si$ of $\mathring{E}_\si$. By isoperimetric inequality we get that $\cH^1(\pa E^i_\si)>0$, and thus there is $y\in \pa E^i_\si \cap \cF E_\si \cap \cF E_\ga$. Since $E^i_\si\cap \mathring{E}_\ga=\epty$, then $\nu_{E_\si}(y)=-\nu_{E_\ga}(y)$ and we conclude that $\mathring{E}_\ga=(E_\si)^c$. But both $E_\si$ and $E_\ga$ are bounded, so this gives a contradiction, and thus $E_\ga=E_\si$.\\
	Now we prove that $V=V_E$. Let
	\beqs
		X=\{p\in\Ga\,|\, \te_V(p)=1,\,\,\ga_i(t)=p\,\Rightarrow\, \ga_i \mbox{ is differentiable at } t \}.
	\eeqs
	We want to prove that
	\beq \label{ref10}
		\cH^1\big( \cF E \De X  \big)=0,
	\eeq
	which implies that $V=V_E$.\\
	If $\ga_i(t)=p\in X$, then there is $\ep>0$ such that $\ga_i\big((t-\ep,t+\ep)\big)\con \{\te_V=1\}\con \Ga=\pa E$. By Rademacher we therefore have that $\cH^1 (X \cap \ga_i\big((t-\ep,t+\ep)\big) \sm \cF E)=0$. Hence $\cH^1(X\sm\cF E)=0$.\\
	Now let $p\in \cF E$, we want to prove that $\cH^1(\cF E\sm X)=0$, and this will complete the claim \eqref{ref10}. If $\te_V(p)=1$ only a curve passes (once) trough $p$, say $\ga_1(t_1)=p$, and since $p\in\cF E$ such curve has to be differentiable at $t_1$. %altrimenti si hanno contraddizioni, vedi conti a pag 2 dei fatterelli
	Conversely if $p=\ga_i(t_i)$ for some $\{i,t_i\}$'s, assuming that each $\ga_i$ is differentiable at $t_i$, we want to prove that $\te_V(p)=1$. Suppose by contradiction that $\te_V(p)>1$, then there are $\al,\be:(-\ep,\ep)\to\Ga$ Lipschitz different arcs such that $\al(0)=\be(0)=p$ and $\al,\be$ are differentiable at time $0$; moreover the hypothesis $\cH^1(\{x\,|\,\te_V(x)>1 \})=0$ implies that $\cH^1\big((\al)\cap(\be)\big)=0$. Therefore $\cH^1$-ae point $p\in (\al)\cup(\be)$ is contained in $X$, and thus $\cH^1$-ae point $p\in (\al)\cup(\be)$ is contained in $\cF E$, since we already know that $\cH^1(X\sm\cF E)=0$. So for any $\ep>0$ there is $r>0$ such that
	\beqs
		\cH^1\big( [(\al)\cup (\be)]\cap B_r(p) \big)\ge (2-\ep)2r,
	\eeqs
	and thus
	\beqs
		\cH^1(\cF E\cap B_r(p))\ge \cH^1\big( [(\al)\cup (\be)]\cap B_r(p) \big)\ge (2-\ep)2r,
	\eeqs
	which is a contradiction with the fact that any point in $\cF E$ has one dimensional density equal to $1$.\\
	So we have proved that a point $p\in\cF E$ verifies that: if $\te_V(p)=1$ then $p\in X$, and if any curve passing through $p$ at some time is differentiable at that time then $p\in X$. In any case we conclude that $\cH^1$-almost every point in $\cF E$ belongs to $X$, and then $\cH^1(\cF E\sm X)=0$.
\end{proof}

%%%%%%%%%%%%%%%%%%%%%%%%%%%%%%%%%%%%%%%%%%%%%%%%%%%%%%%%%%%%%%%%%%%%%%%%%%%%
% RELAXATION

\section{Relaxation} \label{relaxation}

\subsection{Setting and results} From now on and for the rest of Section \ref{relaxation} let $p>1$ be fixed. For any measurable set $E\con \R^2$ we define the energy
\beq \label{ref6}
	\cF_p(E)=\begin{cases}
		\mu_{V_E}(\R^2)+\cE_p(V_E)  & \begin{split}
							\mbox{if }  \,&V_E=\sum_{i\in I} (\ga_i)_\sharp (\bv(S^1,1)),\quad\ga_i:S^1\to\R^2 \,\,C^2\mbox{-immersion}, \\ &\sharp I<+\infty,		
		\end{split}\\
		+\infty & \mbox{otherwise.}
	\end{cases}
\eeq
We write $\cF_p(E)$ understanding that $\cF_p$ is defined on the set of equivalence classes of characteristic functions endowed with $L^1$ norm. We want to calculate the relaxed functional $\overline{\cF_p}$ with respect to the $L^1$ sense of convergence of characteristic functions.\\
By Remark \ref{rem1} and Lemma \ref{lem2}, if $\fp(E)<\infty$, we have that
\beqs
	\cH^1(\pa E\sm\cF E)=0, \qquad \fp(V_E)=\sum_{i\in I} \fp (\ga_i),
\eeqs
if $V_E=\sum_{i\in I} (\ga_i)_\sharp (\bv(S^1,1))$. Also up to replacing $E$ with its complement, we can suppose that $E$ is bounded.\\

\noindent If $E\con \R^2$ is measurable, we define
\beqs
\begin{split}
	\cA(E)=\bigg\{ V=\bv(\Ga,\te_V)= \sum_{i\in I} (\ga_i)_\sharp (\bv(S^1,1)) \,\,\bigg|\,\, & \ga_i:S^1\to\R^2 \,\,C^{1}\cap W^{2,p}\mbox{-immersion},\,\, \sharp I<+\infty, \\ &  \sum_{i\in I}\fp(\ga_i)<+\infty, \\
	& \pa E\con \Ga,\,\, V_E \le V,\\ &  \cF E\con\{ x\in\R^2\,|\,\te_V(x) \mbox{ is odd}\}, \\ & \cH^1( \{ x \,|\, \te_V(x) \mbox{ is odd} \} \sm \cF E )=0  \bigg\},
\end{split}
\eeqs

\begin{remark}
	Observe that if $V\in\cA (E)$, then $\fp(V)<+\infty$, and then $\te_V(x)=\lim_{\ro \searrow 0 }\frac{\mu_V(B_\ro(x))}{2\ro}$ exists and it is uniformly bounded on $\Ga$. Moreover the condition $\pa E\con \Ga$ and the bound on the energy of the curves imply that $\cH^1(\pa E)<\infty$, and then $E$ is a set of finite perimeter.\\
\end{remark}

\noindent The main result of the section is the following.
\begin{thm} \label{thmmain}
	For any measurable set $E\con\R^2$ we have that
	\beq \label{main}
		\rfp(E)=\begin{cases}
					+\infty &  \begin{split}\mbox{if }\, 
						& \cA(E)=\epty \,\,\mbox{ or }\,\,E \mbox{ is ess. unbounded},
					\end{split}\\
					\min\big\{ \fp(V)\,\,|\,\,V\in\cA(E) \big\} & \mbox{otherwise},
		\end{cases}
	\eeq
	where we say that a set $E$ is essentially unbounded if $|E\sm B_r(0)|>0$ for any $r>0$.
\end{thm}

\noindent The proof of Theorem \ref{thmmain} will be completed in Subsection \ref{subsproofmain}.\\

\begin{remark}
	Choosing for a measurable set $E$ the $L^1$ representative defined in \eqref{ax}, then the set $E$ is essentially unbounded if and only if it is unbounded. So in the statement of Theorem \ref{thmmain} one can actually write unbounded in place of essentially unbounded.\\
\end{remark}

\begin{remark}
	The characterization of $\rfp$ given by Theorem \ref{thmmain} immediately implies the stability property that
	\beq
		\fp(E)<+\infty \qquad\Rightarrow\qquad\rfp(E)=\fp(E)<+\infty.
	\eeq
	In fact if $\fp(E)<+\infty$, then $V_E\in \cA(E)$. Consider any $W=\bv(\Ga,\te_W)\in\cA(E)\sm\{V_E\}$, then by definition we have that $V_E\le V$ in the sense of measures and $\cF E\con \{x\,|\,\te_W(x) \mbox{ is odd}\}$, and this implies that $\cH^1(\cF E\sm \Ga)=0$. Therefore $\mu_W(\R^2)\ge \cH^1(\cF E)=\mu_{V_E}(\R^2)$, and also $\cE_p(W)\ge \cE_p(V_E)$ by locality of the generalized curvature (\cite{LeMa09}).\\
\end{remark}

\noindent We conclude this part showing some properties of varifolds $V\in\cA (E)$.

\begin{lemma} \label{lemindici2}
	Let $E\con\R^2$ be a bounded set of finite perimeter. Let $V=\bv(\Ga,\te_V)=\sum_{i=1}^N (\ga_i)_\sharp(\bv(S^1,1))$ with $\ga_1,...,\ga_N:S^1\to\R^2$ Lipschitz curves. Suppose that $\cF E\con\Ga$ and
	\beqs
		\cH^1(\cF E \,\De\, \{x\,\,|\,\,\te_V(x) \mbox{ is odd}\})=0.
	\eeqs
	Then
	\beqs
		E=\bigg\{ p\in\R^2\sm\Ga\,:\, \bigg|\sum_{i=1}^N \ind_{\ga_i}(p)\bigg| \mbox{ is odd}  \bigg\}.
	\eeqs
\end{lemma}

\begin{proof}
	Fix $p\in\R^2\sm \Ga$. In the following we suppose without loss of generality that $p=0$. By hypotheses and by the calculations in the proof of Lemma \ref{lemV=V_E}, there exists a vector $v\in\R^2\sm\{0\}$ such that the ray $L=\{p+tv\,|\,t\in[0,\infty)\}$ verifies the properties:\\
	i) $L$ intersects $\Ga$ at points $y$ such that for any $i=1,...,N$ if $\ga_i(t)=y$ then $\ga_i$ is differentiable at $t$,\\
	ii) $L$ intersects $\cF E$ a finite number $M\in\N$ of times at points $z$ in $\cF E \cap \{x\,\,|\,\,\te_V(x) \mbox{ is odd}\}$ where $\nu_E(z),v$ are independent,\\
	iii) $L$ intersects $\Ga\sm\cF E$ a finite number of times at points $w$ in $ \{x\,\,|\,\,\te_V(x) \mbox{ is even}\}$ where $\ga_i'(t),v$ are independent whenever $\ga_i(t)=w$,\\
	iv) \beqs
			\begin{split}
				\bigg(\sum_{i=1}^N \ind_{\ga_i}(p)\bigg) \qquad \mbox{mod }2 &= \bigg(\sum_{i=1}^N \sum_{y\in L\cap (\ga_i)} \sharp\bigg(\frac{\ga_i}{|\ga_i|}\bigg)^{-1}\bigg(\frac{y}{|y|}\bigg)\qquad\mbox{mod }2\bigg)  \qquad\mbox{mod }2=\\
				&= \bigg(\sum_{i=1}^N \sum_{y\in L\cap (\ga_i) \cap \cF E} \sharp\bigg(\frac{\ga_i}{|\ga_i|}\bigg)^{-1}\bigg(\frac{y}{|y|}\bigg)\qquad\mbox{mod }2\bigg)  \qquad\mbox{mod }2,
			\end{split}
		 \eeqs
	where in iv) the second inequality follows from ii) and iii).\\
	By hypothesis we have that $\partial E \con \Gamma$, and then we can assume that $E=E^1$ is open. Now if $p\in E$, since $E$ is also bounded, the number $M$ has to be odd, and then $\big(\sum_{i=1}^N \ind_{\ga_i}(p)\big) \,\, \mbox{mod }2=1$. Conversely if $p$ is in the interior of $E^c$, then $M$ is even, and then $\big(\sum_{i=1}^N \ind_{\ga_i}(p)\big) \,\, \mbox{mod }2=0$.
\end{proof}

\begin{remark} \label{remindici2}
	We observe that Lemma \ref{lemindici2} applies to couples $E, V$ with $V\in \cA (E)$.\\
\end{remark}

\begin{lemma} \label{lemalternativa}
	Let $V=\bv(\Ga,\te_V)\in\cA(E)$ for some measurable set $E$. Letting $\Si:=\overline{\Ga\sm \pa E}$, it holds that if $\Si \neq \epty$ then for any $x\in \Si\cap \pa E$ at least one of the following holds:
	\beq \label{P}
	\begin{split}
		&i) \,\,\exists y\in \Si\cap \pa E,\,\,\exists\,f:[0,T]\to\R^2\,\,C^1\cap W^{2,p},\,\,T> 0, \\ & \quad f\mbox{ regular curve from }x \mbox{ to }y \mbox{ with } (f)\con \Ga,\\
		& ii) \,\, x \mbox{ is not isolated in }\Si\cap \pa E.
	\end{split}
	\eeq
	The alternative above is not exclusive.
\end{lemma}

\begin{proof}
	Write $V=\sum_{i=1}^N (\si_i)_\sharp(\bv(S^1,1))$. Assume $\Si\neq \epty$, that is equivalent to $\Ga\sm\pa E=:S\neq \epty$. Suppose $x\in \Si\cap \pa E$ is isolated in $\Si\cap \pa E$, then we want to prove that condition $i)$ in \eqref{P} holds true. There exists $r_0>0$ such that $B_r(x)\cap \Si\cap \pa E = \{x\}$ for any $r\le r_0$. Up to reparametrization we can say that $\si_1|_{(-\ep,\ep)}:(-\ep,\ep)\to B_{r_0}(x)$ passes through $x$ at time $0$. Up to reparametrize $\si_1(t)$ into $\si_1(-t)$, we can say that there exists a time $T>0$ such that $\si_1|_{(0,T)}\con S$ and $y:=\si_1(T)\in \pa S=\Si\cap \pa E$, looking at $S$ as topological subspace of $\Si$; in fact otherwise $x$ would not be isolated in $\pa S= \Si\cap \pa E$. Defining $f(t)=\si_1(t)$ for $t\in[0,T]$ gives alternative $i)$ in \eqref{P}.\\
\end{proof}

\subsection{Necessary conditions} \label{neccond}
Here we prove that a set $E\con\R^2$ with $\rfp(E)<+\infty$ has the necessary properties that inspire formula \eqref{main}. \\

\noindent Let $E_n$ be any sequence of sets such that $\fp(E_n)\le C$ and $\chi_{E_n}\to \chi_E$ in $L^1(\R^2)$. Let us adopt the notation $V_{E_n}=\sum_{i\in I_n } (\ga_{i,n})_\sharp (\bv(S^1,1))=\bv(\Ga_n,\te_{V_{E_n}})$, so that $\fp(E_n)= \sum_{i\in I_n } \cH^1(\ga_{i,n}) + \cE_p(\ga_{i,n})$. Using also \eqref{ineq1} we have that $0<c\le\cH^1(\ga_{i,n})\le C<\infty$ for any $i,n$. Also $\cE_p(\ga_{i,n})\ge c >0$ for any $i,n$, thus $\sharp I_n<+\infty$ for large $n$ and then we can suppose that $I_n=I$ for any $n$. Also we can choose $E_n$ bounded and by $L^1$ convergence we have that $|E|<+\infty$.\\

\noindent Moreover we observe that in order to calculate the relaxation of $\fp$ we can suppose that the sequence $E_n$ is actually uniformly bounded, hence getting that $E$ is bounded.\\
Indeed if (up to subsequence) we have that for example $\ga_{1,n}\cap B_n(0)^c\neq \epty$, then by boundedness of the length we have $\ga_{1,n}\con (B_{n-c}(0))^c$ for any $n$ for some $c$. Let $\La_n$ be the connected component of $\cup_{i\in I} (\ga_i)$ containing $(\ga_1)$. The component $\La_n$ is equal to some union $\cup_{j\in J_n} (\ga_{j,n})$. Up to relabeling we can suppose that $J_n=J$ for any $n$. Since the length of each curve is uniformly bounded, then there exist open sets $U_n$ such that $\La_n\con U_n$, $U_n\cap \big( \cup_{i\in I\sm J} (\ga_{i,n}) \big)=\epty$, and $U_n\cap B_{R_n}(0)=\epty$ for some sequence $R_n\to \infty$. Therefore the set $E_n':=E_n\sm U_n$ still converges to $E$ in $L^1(\R^2)$, and $\fp(E_n')<\fp(E_n)$.\\ 

\noindent Under the above notation we have the following result.

\begin{lemma} \label{lemmanecess}
	Suppose $E\con\R^2$ verifies that $\rfp(E)<+\infty$. Let $E_n\con\R^2$ be uniformly bounded such that $\chi_{E_n}\to\chi_E$ in $L^1(\R^2)$ with $\fp(E_n)\le C$. Suppose that for any $n$ the set $\{p\,|\,\te_{V_{E_n}}(p)>1\}$ is finite, then any subsequence of $V_{E_n}$ converging in the sense of varifolds converge to an element of $\cA(E)$.
\end{lemma}

\begin{proof}
\noindent The arclength parametrizations $\si_{i,n}$ corresponding to $\ga_{i,n}$ are uniformly bounded in $W^{2,p}$ for any $i\in I_n=I$ and for any $n$. Therefore, since the sequence is uniformly bounded in $\R^2$, up to subsequence $\si_{i,n}\to \si_i$ strongly in $C^{1,\al}$ for some $\al\le\frac{1}{p'}$ and weakly in $W^{2,p}(\R^2)$ for any $i\in I$. Each $\si_i$ is then a closed curve parametrized by arclength, and we call $\ga_i$ the parametrization on $S^1$ with constant velocity.\\
Hence the varifolds $V_{E_n}$ converge to some limit integer rectifiable varifold $V=\bv(\Ga,\te_V)$ in the sense of varifolds, and $V=\sum_{i\in I} (\ga_i)_\sharp(\bv(S^1,1))$. The multiplicity function $\te_V$ is upper semicontinuous and pointwise bounded by the discussion in Subsection \ref{sscmon}. Also the sets $E_n$ converge to $E$ weakly* in $BV(\R^2)$, that is $\chi_{E_n}\to\chi_E$ and $D\chi_{E_n}\wsc D\chi_E$, thus $E$ is a set of finite perimeter. Observe that $|D\chi_{E_n}|=\mu_{V_{E_n}}\wsc \mu_V$.\\

\noindent From now on we call $\Ga=\cup_{i\in I} (\si_i)$, $\Si=\overline{\Ga\sm\pa E}$, $S=\Ga\sm \pa E$.\\

\noindent Let $x\in \pa E$, so that for any $\ro>0$ we have
\beq
\lim_{n} \int_{B_\ro(x)}\chi_{E_n} >0, \qquad \lim_{n} \int_{B_\ro(x)}\chi_{E_n^c} >0.
\eeq
Then for $\ro>0$ there is $n(\ro)$ such that there exist $\xi_n\in E_n\cap B_\ro(x), \,\eta_n\in E_n^c\cap B_\ro(x) $ for any $n\ge n(\ro)$ and thus there exists $w_n\in \pa E_n \cap B_\ro(x) $ for any $n\ge n(\ro)$. Taking some sequence $\ro_k\searrow0$, we find a sequence $w_n$ converging to $x$. Therefore, also by density \eqref{ref2}, we have proved that $\cF E\con \pa E\con \big\{ y\,|\, y=\lim_n y_n,\,y_n\in\cF E_n \big\}=\Ga$. In particular $\pa E$ is $1$-rectifiable.\\

\noindent Now we prove that $\cF E\con \{x\,\,|\,\,\te_V(x) \mbox{ is odd}\}$.\\
So let $p\in\cF E$, and let $\{\ga^j_k\,|\, j=1,...,N, i=1,...,n_j \}$ be distinct curves which are suitable disjoint restrictions of the $\ga_i$'s such that $(\ga^j_k)\con(\ga_j)$ for any $k$ (up to relabeling the $\ga_i$'s) and
\beqs
\Ga\cap B_{r_0}(p)=\bcup_{j,k} (\ga^j_k).
\eeqs
Without loss of generality we write $\ga^j_k(t^j_k)=p$. We want to prove that $\sum_{j=1}^N n_j=\te_V(p)$ is odd.
%Recalling the definition of flux property (Definition \ref{defflux}), which is satisfied by the varifold $V$, since we already know that $\cH^1\big((\Ga\sm\cF E)\De \{x\,\,|\,\,\te_V(x) \mbox{ odd}\}\big)=0$, we get that the number of curves $\ga^j_k$ such that $(\ga^j_k)'(0)$ is not horizontal is even.
Since $p\in\cF E$ there is $q\in{E}\cap B_{r_0}(p)$ such that the segment
\beqs
s(t)= q+\frac{p-q}{|p-q|}t \qquad t\in [0,2|p-q|]
\eeqs
is such that
\beq \label{ref20}
\bigg|\bigg\lgl \frac{p-q}{|p-q|}, (\ga^j_k)'(t^j_k) \bigg\rgl\bigg|>0,
\eeq
and $s|_{[0,|p-q|)}\con E$, $s|_{[|p-q|,2|p-q|]}\con E^c$. Also since $\ga_{i,n}\to\ga_i$ strongly in $C^{1,\al}$, by \eqref{ref20} we get that $s$ intersects transversely $\ga_{i,n}$ for any $i$ for $n$ big enough, and the number of such intersections is $\te_V(p)$. Also denote $b:=s(2|p-q|)$. Moreover we can write that $B_{r_q}(q)\con{E}_n$ and $B_{r_b}(b)\con E_n^c$ for $n$ sufficiently big.\\
We know that for any $\ep>0$ there is $a_\ep\in E_n^{c*}$, where $(\cdot)^*$ will always denote the unbounded connected component of $(\cdot)$, such that
\beqs
\bigg|\frac{p-q}{|p-q|}-\frac{a_\ep-b}{|a_\ep-b|}\bigg|<\ep,
\eeqs
\beqs
\sum_{i\in I}\ind_{\ga_{i,n}}(b) \quad\mbox{mod }2 = \sum_{i\in I}\sharp\bigg(\frac{\ga_{i,n}}{|\ga_{i,n}|}\bigg)^{-1}\bigg(\frac{a_\ep-b}{|a_\ep-b|}\bigg)\quad\mbox{mod }2.
\eeqs
Hence up to a small $C^\infty$ deformation which is different from the identity only on $\big\{ x + t\frac{p-q}{|p-q|} |\,\,x\in B_{r_b}(b) , t\in \R_{\ge0}\big\}\sm B_{r_b}(b)$ we can suppose that for $M>0$ sufficiently big it holds that
\beqs
a_0:=b+M \frac{p-q}{|p-q|}\in E_n^{c*},
\eeqs
\beqs
\bigg\{  b+\R_{\ge0}\bigg(\frac{p-q}{|p-q|} \bigg) \bigg\}\cap \bigg( \bcup_{i\in I} (\ga_{i,n}) \bigg)\con \cF E_n,
\eeqs
\beq \label{ref21}
\sum_{i\in I}\ind_{\ga_{i,n}}(b) \quad\mbox{mod }2 =  \sum_{i\in I}\sharp\bigg(\frac{\ga_{i,n}}{|\ga_{i,n}|}\bigg)^{-1}\bigg(\frac{a_0-b}{|a_0-b|}\bigg)\quad\mbox{mod }2.
\eeq
Taking into account Lemma \ref{lemV=V_E}, by construction we have that the quantity in \eqref{ref21} is $0 \mbox{ mod }2$. Moreover we have that
\beqs
1\quad\mbox{mod }2=\sum_{i\in I}\ind_{\ga_{i,n}}(q) \quad\mbox{mod }2 =\bigg(\te_V(p) + \sum_{i\in I} \ind_{\ga_{i,n}}(b)\bigg)\quad\mbox{mod }2,
\eeqs
and then $\te_V(p)$ is odd.\\

\noindent It remains to prove that $\cH^1(\{x\,\,|\,\,\te_V(x) \mbox{ odd}\}\sm\cF E)=0$.\\
We observe that in the sense of currents we have the convergence $[|E_n|]\to [|E|]$ and thus
\beqs
	\pa [|E_n|] = \boldsymbol{\tau}\bigg(\bcup_{i\in I} (\si_{i,n}), 1, \xi_0 \bigg)  \to \pa [|E|] 
\eeqs
in the sense of currents where $\xi_0$ is the positive orientation of the boundaries with respect to $\R^2$. We can write $\pa [|E_n|]=\sum_{i=0}^\infty (\al_{i,n})_\sharp([|S^1|])$ for countably many Lipschitz parametrizations $\al_{i,n}$ ordered so that $L(\al_{i+1,n})\le L(\al_{i,n})$ for any $i,n$. Such immersions positively orient the boundary $\pa E^i_n$ of $E^i_n$, where $E^i_n$ is one of the open connected components of ${E}_n$, which are at most countable. The length of each $\al_{i,n}$ is uniformly bounded, then we can assume that the parametrizations $\al_{i,n}$ are $L$-Lipschitz with constant $L$ independent of $i,n$. Since the parametrizations $\si_{i,n}$ converge strongly in $C^1$, the immersions $\al_{i,n}$ uniformly converge to $L$-Lipschitz curves $\al_i:S^1\to \R^2$ as $n\to\infty$. We can also reparametrize each $\al_i$ by constant velocity almost everywhere (in the sense of metric derivatives). In the sense of currents we have that
\beqs
	\sum_{i=0}^\infty (\al_{i,n})_\sharp([|S^1|])=\pa[|E_n|] \to \pa[|E|]=  \boldsymbol{\tau}\big(\cF E, 1, \xi_0 \big).
\eeqs
Let us define
\beqs
	T:=\sum_{i=0}^\infty  (\al_i)_\sharp([|S^1|]).
\eeqs
Since each $(\al_{i,n})$ is contained in some $(\si_{{i_0},n})$ we have that $d_\cH (\alpha_{i,n},\alpha_i) \le N \max_{i=1,...,N} \|\si_{i,n} - \si_i\|_{\infty} \le \ep$ for any $n\ge n_\ep$.
%If $d_\cH$ is the Hausdorff distance, we can estimate
%\beqs
%\begin{split}
%	d_\cH\big((\al_{i,n}),(\al_{i,m})\big)&\le d_\cH\big((\al_{i,n}),(\si_{{i_0},m})\big)\le d_\cH\big((\si_{{i_0},n}),(\si_{{i_0},m})\big)=\\&= \max\bigg\{ \sup_t\inf_s \big| \si_{{i_0},n}(t)-\si_{{i_0},m}(s) \big|, \sup_t\inf_s \big| \si_{{i_0},m}(t)-\si_{{i_0},n}(s) \big| \bigg\}\le\\
%	&\le \max_{i\in I}\sup_t  \big| \si_{i,n}(t)-\si_{i,m}(t) \big|.
%\end{split}
%\eeqs
%Hence for any $\ep>0$ there is $n_\ep$ such that $d_\cH\big((\al_{i,n}),(\al_i)\big)<\ep$ for any $n\ge n_\ep$.
Since an equivalent definition of Hausdorff distance is $d_\cH(A,B)=\inf\big\{\ep>0\,\,|\,\,A\con\cN_\ep(B),\,B\con\cN_\ep(A)\big\}$ where $\cN_\ep(X)=\{x\,\,|\,\,d(x,X)\le\ep\}$, we have that
\beq\label{eq1}
	\forall\ep>0\,\exists n_\ep\,:\qquad d_\cH\big(\cup_i(\al_{i,n}),\cup_i(\al_i)\big)<\ep \qquad n\ge n_\ep.
\eeq
Thus $\cup_i(\al_{i,n})$ converges in Hausdorff distance to the set $\overline{\cup_i(\al_i)}$. Moreover, writing $\cup_i(\al_{i,n})=\sqcup_1^{k_n} C^j_n$ as a disjoint union of finitely many compact connected components, by a diagonal argument, applying Go\l ab Theorem on each component, we can assume $k_n=k$ for any $n$ and that $C^j_n$ converges in Hausdorff distance to a compact connected set $C^j$ for any $j=1,...,n$. Therefore $\overline{\cup_i(\al_i)}=\cup_1^k C^j = \Ga$, and then $\cH^1(\overline{\cup_i(\al_i)})=\cH^1(\Ga)$ is finite and $\overline{\cup_i(\al_i)}$ is closed and $1$-rectifiable.\\

%Thus $\cup_i(\al_{i,n})$ converges in Hausdorff distance to the set $\cup_i(\al_i)$. By hypothesis the set $\{p\,\,|\,\,\te_{V_{E_n}}>1\}$ is finite, hence $\cup_i(\al_{i,n})=\pa E_n$ is closed, bounded, and has finitely many connected components for any $n$. The connected components of $\pa E_n$ are at most $\sharp I<+\infty$, thus applying Golab Lower Semicontinuity Theorem on each component, we have that $\cup_i(\al_i)$ is closed as well. Moreover
%\beq\label{eq2}
%	\Ga=\cup_i(\al_i).
%\eeq
%In fact it is easy to see that $\Ga\supset\cup_i(\al_i)$, and if by contradiction $\Ga\supsetneq \cup_i(\al_i)=\overline{\cup_i(\al_i)}$, then by \eqref{eq1} there is $r>0$ such that $\epty=B_{r}(p)\cap \big(\cup_i(\al_{i,n})\big)=B_r(p)\cap\big(\cup_{i\in I}(\ga_{i,n})\big)$ for any $n$ big enough, that is impossible.\\

\noindent Let $x\in \R^2\sm\Ga$. By \eqref{eq1} we have that there is $\ro>0$ such that $B_\ro(x)\cap \big(\overline{\cup_i(\al_i)}\cup\cup_i(\al_{i,n})\big) =\epty$ for any $n$ large. Then there exists $n_x$ such that for any $i$ the index $\ind_{\al_{i,n}}(x)$ is the same for any $n\ge n_x$.\\
In fact suppose by contradiction for any $n$ there is $i_n,N_1,N_2\ge n$ with $1=\ind_{\al_{{i_n},N_1}}(x)\neq\ind_{\al_{{i_n},N_2}}(x)=0$ and $i_n\to\infty$ without loss of generality. Then $L(\al_{{i_n},N_1})\ge C(\ro)$ for a constant $C(\ro)>0$ depending only on $\ro$ by isoperimetric inequality. Since $L(\al_{i+1,n})\le L(\al_{i,n})$ for any $i,n$ and $i_n\to\infty$, this implies $P(E_n)$ is arbitrarily big that for $n$ large enough.\\

\noindent Now let $x\in \R^2\sm\Ga$ such that there exists $\lim_n\chi_{E_n}(x)$. Since $\chi_{E_n}(x)=\sum_i \ind_{\al_{i,n}}(x)$ for $n$ big such that $B_\ro(x)\cap \big(\overline{\cup_i(\al_i)}\cup\cup_i(\al_{i,n})\big) =\epty$ for some $\ro>0$, from the above discussion we have that
\beqs
\begin{split}
	\lim_n \sum_i \ind_{\al_{i,n}}(x)=1 \quad &\Leftrightarrow \quad \forall n\ge n_0 \,\exists i\,: \quad \ind_{\al_{i,n}}(x)=1 \\
	&\Leftrightarrow \quad \exists i(x)\, \forall n\ge n_0 \quad \ind_{\al_{i(x),n}}(x)=1\\
	&\Leftrightarrow \quad \exists i(x)\,:\quad\ind_{\al_{i(x)}}(x)=1.
\end{split}
\eeqs
Hence
\beqs
\begin{split}
	x\in E  \quad \Leftrightarrow \quad \lim_n \sum_i \ind_{\al_{i,n}}(x)=1
	\quad \Leftrightarrow \quad  \exists i(x)\,:\quad\ind_{\al_{i(x)}}(x)=1 
	\quad \Leftrightarrow \quad \sum_i \ind_{\al_i}(x)=1.
\end{split}
\eeqs
In particular
\beq\label{eq3}
	E=\bigg\{x\in\R^2\sm\Ga\,\,|\,\, \sum_{i=0}^\infty \ind_{\al_i}(x)=1\bigg\}= \bigg\{x\in\R^2\sm\Ga\,\,:\,\, \bigg|\sum_{i\in I} \ind_{\si_i}(x)\bigg| \mbox{ is odd}\bigg\},
\eeq
up to $\cL^2$-negligible sets, where the second equality follows by the uniform convergence of the finitely many curves $\si_{i,n}$. Also for any $i\neq j$ it holds that $|\{x\in\R^2\sm(\al_i)\,\,|\,\,\ind_{\al_i}(x)=1\}\cap \{x\in\R^2\sm(\al_j)\,\,|\,\,\ind_{\al_j}(x)=1\}|=0$, because the equality holds for any $n$ for $\al_{i,n},\al_{j,n}$. Therefore we have
\beqs
\begin{split}
	\sum_i \int_{(\al_{i,n})} \lgl \om, \tau_{i,n}\rgl &= \int_{E_n} d\om \xrightarrow[n]{} \int_{\big\{\sum_{i=0}^\infty \ind_{\al_i}(x)=1\big\}} d\om=\\& = \sum_i \int_{\big\{\ind_{\al_i}(x)=1\big\}} d\om = \sum_i \int_{(\al_i)} \lgl \om,\tau_i\rgl,
\end{split}
\eeqs
for any $1$-form $\om$ on $\R^2$. This means that
\beq
	\sum_{i=0}^\infty (\al_{i,n})_\sharp([|S^1|])=\pa[|E_n|]\to T=\sum_{i=0}^\infty  (\al_i)_\sharp([|S^1|])=\pa[|E|]=  \boldsymbol{\tau}\big(\cF E, 1, \xi_0 \big),
\eeq
in the sense of currents. In particular we can write the multiplicity function of the current $\pa[|E|]$ as
\beq \label{ref5}
	m(x)=\sum_{i=0}^\infty \sum_{y\in \al_i^{-1}(x)} S(y),
\eeq
for $\cH^1$-ae $x\in\R^2$, where $S(y)=+1$ if $d (\al_i)_{y}$ preserves the orientation and $S(y)=-1$ in the opposite case. Note that since $\te_V$ is bounded, $\Ga=\overline{\cup_i(\al_i)}$, and $\te_V(p)\ge\sum_i\sharp\al_i^{-1}(p)$, then the series in \eqref{ref5} is actually a finite sum.\\
Also observe that since $E$ is a set of finite perimeter, by Gauss-Green formula the multiplicity function $m$ is equal to $1$ $\cH^1$-ae on $\cF E$, $\cH^1(\{x\,|\,m(x)\ge 1\}\sm\cF E)=0$, and $m=0$ $\cH^1$-ae on $\R^2\sm \cF E$. Also, $m(x)=0$ at $\cH^1$-ae $x \in \Ga \sm \cup_i (\al_i)$.\\

\noindent Now since $\Ga=\overline{\cup_i(\al_i)}$ and $\cH^1(\al_i(\{t\,:\,\not\exists \al_i'(t)\}))=0$, then
\beq\label{eq4}
	\cH^1\big(  \{p\in\Ga\,\,|\,\,\exists\,t,i\,: \al_i(t)=p, \not\exists \al_i'(t) \} \big)=0.
\eeq
So let $p\in\cup_i(\alpha_i)$ be such that if $\al_i(t)=p$ then $\exists\al_i'(t)$. We want to check that $\te_V(p)$ and $\sum_i\sharp\al_i^{-1}(p)$ have the same parity. In fact if without loss of generality $\te_V(p)>\sum_i\sharp\al_i^{-1}(p)$, taking into account \eqref{eq3}, following a segment $s$ intersecting $\cup_i(\alpha_i)$ only at $p$ and transversally (as in the first part of the proof) we have that:\\
i) $s$ passes from $E$ to $E^c$ if and only if $\te_V(p)$ is odd, or equivalently if and only if $\sum_i\sharp\al_i^{-1}(p)$ is odd;\\
ii) $s$ passes from $E$ to $E$ if and only if $\te_V(p)$ is even, or equivalently if and only if $\sum_i\sharp\al_i^{-1}(p)$ is even.\\
Hence by \eqref{ref5} we conclude that $\te(p)$ is odd if and only if alternative i) above holds, if and only if the summands in \eqref{ref5} are odd, if and only if $m(p)$ is odd.\\
By \eqref{eq4} this holds for $\cH^1$-ae point in $\cup_i(\alpha_i)$. Therefore $\cH^1(\{x\,|\,m(x) \mbox{ is odd}\}\De \{x\,|\,\te_V(x) \mbox{ is odd}\})=0$. So finally since $\cH^1 (\{x\,|\,m(x)=1\}\sm \cF E  )=0$, then
\beqs
\begin{split}
		0&=\cH^1(\{x\,|\,m(x) \mbox{ odd}\}\sm \{x\,|\,m(x)=1\}) =\cH^1(\{x\,|\,\te_V(x) \mbox{ odd}\}\sm \{x\,|\,m(x)=1\}) =\\ &=\cH^1( \{ x \,|\, \te_V(x) \mbox{ is odd} \} \sm \cF E ),
\end{split}
\eeqs
which completes the proof.\\
\end{proof}
%\noindent In conclusion we observe that as a consequence we also have that $E=\overline{\mathring{E}}$, recalling the convention \eqref{ax}. Indeed if $x\in \cF E$, then since each $\ga_i\in C^{1,\al}$ clearly $x\in \overline{\mathring{E}}$; while if $x\in \pa E\sm \cF E =\overline{\cF E}\sm \cF E$, then $x$ is the limit of a sequence in $\cF E$, and by a diagonal argument it is also the limit of a sequence in $\mathring{E}$, hence $x\in \overline{\mathring{E}}$. So finally $\mathring{E}\cup \pa E \con \overline{\mathring{E}} $, and since $\overline{\mathring{E}}\con \overline{E}=E$ is trivial we have the claim.\\

\subsection{Proof of Theorem \ref{thmmain}} \label{subsproofmain}
%Consider $E\con \R^2$ bounded such that $\rfp(E)<+\infty$, and assume without loss of generality that $|E|>0$.
First we want to prove the following approximation result.

\begin{prop} \label{propapprox}
	Let $E\con\R^2$ be measurable and bounded with $\cA(E)\neq\epty$. Then for any $V\in\cA(E)$ there exists a sequence $E_n$ of uniformly bounded sets such that
	\beq
		\fp(E_n)<+\infty, \quad \chi_{E_n}\to\chi_E \quad\mbox{in }L^1(\R^2), \quad  V_{E_n}\to V \, \mbox{as varifolds}, \quad \lim_n \fp(E_n)=\fp(V).
	\eeq
	Moreover for any $n$ we have that $V_{E_n}=\sum_{i=1}^{N} (\ga_i)_\sharp(\bv(S^1,1))=\bv(\Ga_n,\te_{V_{E_n}})$ and $\{p\,|\,\te_{V_{E_n}}(p)>1\}$ is finite.
\end{prop}

\begin{proof}
	Let $V=\sum_{i=1}^N (\ga_i)_\sharp(\bv(S^1,1))\in\cA(E)$ with $\ga_i\in W^{2,p}$ regular. For any $i$ let $\{\ga_{i,n}\}_{n\in\N}$ be a sequence of analytic regular immersions such that $\ga_{i,n}\to\ga_i $ in $W^{2,p}$ as $n\to\infty$. Hence the set
	\beq \label{ref19}
	\{x\in\R^2\,|\, \exists i,j,t\neq \tau: \,\, \ga_i(t)=\ga_j(\tau)  \}
	\eeq
	is finite. Let $V_n=\sum_{i=1}^N (\ga_{i,n})_\sharp(\bv(S^1,1))$. By \eqref{ref19} we can define $E_n$ as in Lemma \ref{lemV=V_E}, so that $V_n=V_{E_n}$. Moreover we have that
	\beqs
	\fp(E_n)<+\infty, \qquad\lim_{n\to\infty} \fp(V_{E_n})=\lim_{n\to\infty} \fp(E_n)=\fp(V), \qquad	V_{E_n}\to V.
	\eeqs
	By uniform convergence of $\ga_{i,n}$ we get that for any $\ep>0$ there is $n_\ep$ such that
	\beqs
	\bigcup_{i=1}^N (\ga_{i,n}) \con I_{\frac{\ep}{2}} \left(\bigcup_{i=1}^N (\ga_i)\right)   \quad\qquad\forall n\ge n_\ep,
	\eeqs
	where $I_{\frac{\ep}{2}}$ denotes the $\frac{\ep}{2}$ open tubolar neighborhood. Hence up to passing to a subsequence by Riesz-Fr\'{e}chet-Kolmogorov we have that $\chi_{E_n}$ converges strongly in $L^2(\R^2)$, and then in $L^1(\R^2)$ and pointwise almost everywhere to the characteristic function of a closed set $F$. Using the definition of $E_n$ and Lemma \ref{lemindici2} together with Remark \ref{remindici2} we have that $F=E$, and the proof is completed.\\
\end{proof}

\begin{cor} \label{cormin}
	Let $E\con\R^2$ be measurable and bounded with $\cA(E)\neq\epty$. Then
	\beqs
		\exists\, \min\big\{ \fp(V)\,\,|\,\,V\in\cA(E) \big\}.
	\eeqs
\end{cor}

\begin{proof}
	Let $V_k$ be a minimizing sequence in $\cA(E)$. Up to subsequence we can assume that $V_k\to V$ in the sense of varifolds and the supports $\supp V_k$ are uniformly bounded. By Proposition \ref{propapprox} using a diagonal argument we find a sequence of uniformly bounded sets $E_k$ such that
	\beqs
		\begin{aligned}[c]
			& \chi_{E_k}\to\chi_E \qquad\mbox{in }L^1(\R^2),\\
			& V_{E_k}\to V \qquad \mbox{as varifolds},
	\end{aligned}
	\qquad\qquad\qquad
	\begin{aligned}[c]
	& \fp(E_k)\le C<+\infty,\\
	& \lim_k \fp(E_k)=\lim_k \fp(V_k)=\inf_{\cA(E)} \fp \ge \fp(V),
	\end{aligned}
	\eeqs
	and $\{p\,|\,\te_{V_{E_k}}(p)>1\}$ is finite. Hence $E_k$ is a possible approximating sequence of $E$ by regular sets, i.e. a competitor in the calculation of the relaxation $\rfp(E)$. Then by Lemma \ref{lemmanecess} we get that $V\in\cA(E)$, and therefore $V$ minimizes $\fp$ on $\cA(E)$.\\
\end{proof}

\noindent Now Proposition \ref{propapprox} together with Corollary \ref{cormin} readily imply Theorem \ref{thmmain}.\\

\subsection{Comment on the $p=1$ case} The characterization of the relaxed energy given by Theorem \ref{thmmain} fails in the $p=1$ case. As stated in Section 1, many estimates used in the $p>1$ case have an analogous formulation in case $p=1$. However, if $I\con\R$ is a bounded interval, functions $u\in W^{2,1}(I)$ do not have good compactness properties. In fact even if $u\in W^{2,1}(I)$ implies that $u'\in W^{1,1}(I)=AC(\bar{I})$ and hence $u\in C^1$, the immersion $W^{2,1}(I)\hookrightarrow C^1(\bar{I})$ is not continuous.\\
Since $W^{2,1}(I)\hookrightarrow W^{1,p}(I)$ for any $p\in[1,\infty)$, we have that $W^{2,1}(I)$ compactly embeds only in $C^{0,\al}(\bar{I})$ for any $\al\in(0,1)$. This implies that the convergence of the curves defining the boundary of sets $E_n$ with $\cF_1(E_n)\le C$ is much weaker than in the $p>1$ case.\\

\noindent One of the main differences is the following. As we will show in Subsection \ref{subsex} the $\cF_p$ energy of polygons is infinite if $p>1$. Instead if $E$ is a regular polygon, i.e. a set $E\con\R^2$ whose boundary is the image of an injective piecewise $C^2$ closed curve, it can happen that $\overline{\cF_1}(E)<+\infty$. For instance, consider a square $Q$ in the plane: in small neighborhoods of the four vertices the boundary $\pa Q$ can be approximated by a piece of circumference of radius converging to $0$ with finite bounded energy converging to $\frac{\pi}{2}$. This is ultimately due to the invariance of the energy $\cF_1$ under rescaling, a property that is absent if $p>1$. This implies that a possible limit varifold does not verify the flux property (because of the arguments in the proof of Proposition \ref{poligoni}).\\

\noindent We believe that the presence of vertices in the boundary of the limit set is the main difference with the $p>1$ case and that sets $E$ with $\overline{\cF_1}(E)\le C$ have at most countably many vertices, each of them giving an additional contribution to the energy equal to the angle described by the vertex.\\

%%%%%%%%%%%%%%%%%%%%%%%%%%%%%%%%%%%%%%%%%%%%%%%%%%%%%%%%%%%%%%%%%%%%%%%%%%%%%%
% ESEMPI, COMMENTI

\section{Remarks and applications}

\subsection{Comparison with \cite{BeMu04}, \cite{BeMu07}}\label{BM} In these works Bellettini and Mugnai develop a characterization of the following relaxed functional. For simplicity we reduce ourselves to the case $p=2$. Let $E\con\R^2$ be measurable and define the energy
\beq
	G(E)=\begin{cases}
		\int_{\pa E} 1 + |k_{\pa E}|^2\, d\cH^1  & E \mbox{ is of class } C^2,\\
		+\infty  & \mbox{otherwise.}
	\end{cases}
\eeq
Then the functional $\overline{G}$ is the $L^1$-relaxation of $G$. Clearly
\beqs
	G(E)<+\infty \quad\Rightarrow\quad \cF_2(E)=G(E),
\eeqs
\beqs
	\rfd(E)\le \overline{G}(E) \qquad\forall E.
\eeqs
The precise characterization of $\overline{G}$ is discussed in \cite{BeMu04} and \cite{BeMu07}; here we just want to point out that
\beqs
	\exists E:\qquad \rfd(E)<\overline{G}(E)<+\infty.
\eeqs
In fact an example is the set $E_0$ in Fig. \ref{figBM} described in the Example 4.4 in \cite{BeMu07}. Let $\ga_1,\ga_2$ be as in Fig. \ref{figBM}. In \cite{BeMu07} it is proved that
\beqs
	\overline{G}(E)>\cF_2(\ga_1)+\cF_2(\ga_2).
\eeqs

\begin{figure}[h]
	\begin{center}
		\begin{tikzpicture}[scale=2]
		\draw
		(0,-0.5)--(0,0.5);
		\draw
		(-0.5,0)--(0.5,0);
		\filldraw[fill=white, pattern=north east lines]
		(0,0.5)to[out=90, in=270, looseness=1](-0.5,1)to[out=90, in=180, looseness=1](0,1.5)to[out=0, in=90, looseness=1](0.5,1)to[out=270, in=90, looseness=1](0,0.5);
		\filldraw[rotate=90, fill=white, pattern=north east lines]
		(0,0.5)to[out=90, in=270, looseness=1](-0.5,1)to[out=90, in=180, looseness=1](0,1.5)to[out=0, in=90, looseness=1](0.5,1)to[out=270, in=90, looseness=1](0,0.5);
		\filldraw[rotate=180, fill=white, pattern=north east lines]
		(0,0.5)to[out=90, in=270, looseness=1](-0.5,1)to[out=90, in=180, looseness=1](0,1.5)to[out=0, in=90, looseness=1](0.5,1)to[out=270, in=90, looseness=1](0,0.5);
		\filldraw[rotate=270, fill=white, pattern=north east lines]
		(0,0.5)to[out=90, in=270, looseness=1](-0.5,1)to[out=90, in=180, looseness=1](0,1.5)to[out=0, in=90, looseness=1](0.5,1)to[out=270, in=90, looseness=1](0,0.5);
		\path[font=\normalsize]
		(-1,1)node[left]{$E_0$};
		\path[font=\normalsize]
		(0.75,1.3)node[left]{$\ga_2$};
		\path[font=\normalsize]
		(1.3,-0.6)node[left]{$\ga_1$};
		\path[font=\tiny]
		(0.05,0.5)node[left]{$(0,1)$};
		\path[font=\tiny]
		(0.7,0.1)node[left]{$(1,0)$};
		\path[font=\tiny]
		(0.5,-0.5)node[left]{$(0,-1)$};
		\path[font=\tiny]
		(-0.1,-0.1)node[left]{$(-1,0)$};
		\end{tikzpicture}
	\end{center}
	\caption{Picture of the set $E_0$ in Example 4.4 of \cite{BeMu07}. The curve $\ga_1$ parametrizes the left and the right components, while $\ga_2$ parametrizes the upper and lower components. The varifold $(\ga_1)_\sharp(\bv(S^1,1))+(\ga_1)_\sharp(\bv(S^1,1))$ belongs to $\cA(E_0)$, and it has multiplicity equal to $1$ on $\pa E_0$ and equal to $2$ on the cross in the middle of the picture.}\label{figBM}
\end{figure}
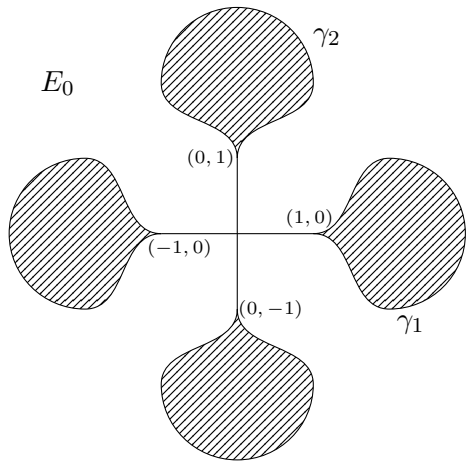

\noindent Here we want to prove that
\beq \label{ref14}
	\rfd(E)= \cF_2(\ga_1)+\cF_2(\ga_2).
\eeq
Observe that $\ga_1,\ga_2$ carry inside $B_1(0)$ a $\cF_2$ energy equal to $8$.\\
Since $\rfd(E_0)<+\infty$ there exists a varifold $V=\sum_{i=1}^N (\ga_1)_\sharp(\bv(S^1,1))\in\cA (E)$. Up to renaming and reparametrization assume $\ga_1(0)=(1,0)$, $\ga_1'(0)=-(1,0)$, and $\ga_1|{[-T,0]}$ joins $(0,1)$ and $(0,1)$ having support contained in $\cF E_0\sm \overline{B_1(0)}$. Since $\ga_1$ is $C^1$ and closed, by the above discussion there exists a first time $\tau>0$ such that $\ga_1(\tau)\in\{(1,0),(0,1),(-1,0),(0,-1)\}$. We divide two cases.\\
1) If $\ga_1(\tau)\in \{(0,1),(0,-1)\}$, arguing like in the proof of inequality \eqref{ineq1} one has
\beqs
	\frac{\pi}{2}\le \big[L(\ga_1|_{(0,\tau)})\big]^{\frac{1}{2}}\big[\cE_2(\ga_1|_{(0,\tau)})\big]^{\frac{1}{2}}\le \frac{1}{2}\big( L(\ga_1|_{(0,\tau)}) + \cE_2(\ga_1|_{(0,\tau)}) \big),
\eeqs
then $\cF_2(\ga_1|_{(0,\tau)})\ge\pi>2$.\\
2)If $\ga_1(\tau)=(1,0)$ by an analogous argument one gets $\cF_2(\ga_1|_{(0,\tau)})\ge 2\pi>2$.\\
Hence in any case it is convenient for the curve $\ga$ to pass first through the point $(-1,0)$. By the characterization of Theorem \ref{thmmain} equality \eqref{ref14} follows.\\

\noindent In this sense we can look at the relaxation $\rfp$ as a generalization of the energy $\overline{G}$, in the sense that $\cF_p$ admits a wider class of regular objects, i.e. sets $E$ with $\cF_p(E)<+\infty$, and this implies that the relaxed energy $\rfp$ is naturally strictly less than $\overline{G}$ on some sets.\\

\subsection{Inpainting}

Here we describe a simple but significant application of the relaxed functional $\rfp$ given by Theorem \ref{thmmain}. Such application arises from the inpainting problem that roughly speaking consists in the reconstruction of a part of an image, knowing how the remaining part of the picture looks like. This problem as stated is quite involved (\cite{BeCaMaSa11}). Assuming the only two colours of the image are black and white, as already pointed out for example in \cite{AmMa03}, one can think that the black shape contained in lost part of the image is consistent with the shape minimizing a suitable functional depending on length and curvature of its boundary. In such a setting the known part of the image plays the role of the boundary conditions. On different scales one can ask for the optimal unknown shape to minimize a weighted functional like \eqref{ref11}, where one can give more importance to the length or to the curvature term.\\
Now we formalize the problem and we give a variational result.\\

\noindent Fix $p\in (1,\infty)$. In $\R^2$ consider the set $H$ defined as follows. Let $Q_1,Q_2$ be the squares $Q_1=\{(x,y):0\le x\le 10,0\le y\le 10\},Q_2=\{(x,y):-10\le x\le 0,-10\le y\le 0  \}$, modify the squares in small neighborhoods of the vertices into convex sets $\tilde{Q}_1,\tilde{Q}_2$ with smooth boundary. Finally let
\beq
	H=\big(\tilde{Q}_1\cup \tilde{Q}_2\big)\sm B_1(0).
\eeq
Let $\la\in(0,\infty)$ and $\flp$ be the functional
\beq \label{ref11}
\flp(E)=\begin{cases}
	\la\mu_{V_E}(\R^2)+\cE_p(V_E)  & \begin{split}
		\mbox{if }  \,&V_E=\sum_{i\in I} (\ga_i)_\sharp (\bv(S^1,1)),\quad\ga_i:S^1\to\R^2 \,\,C^2\mbox{-immersion}, \\ &\sharp I<+\infty,		
	\end{split}\\
	+\infty & \mbox{otherwise.}
\end{cases}
\eeq
Analogously to the functional $\cF_p$, we have a well defined characterization of the $L^1$-relaxed functional $\rflp$.\\
We want to solve the minimization problem
\beq \label{pbinp}
	\mathfrak{P} = \min\big\{ \rflp(E) \,\,\big|\,\, E\con\R^2 \mbox{ measurable s.t. }E\sm B_1(0)=H \big\},
\eeq
under the hypothesis of $\la$ suitably small. The heuristic idea is that a good candidate minimizer is given by the set
\beq
	E_0=\big[(Q_1\cup Q_2)\cap \overline{B_1(0)}\big]\cup H,
\eeq
which has finite $\cF_p$ energy. For a qualitative picture see Fig. \ref{figinp}.

\begin{figure}[h]
	\begin{center}
		\begin{tikzpicture}[scale=0.8]
		\draw[color=black,scale=1,domain=-3.141: 3.1412,
		smooth,variable=\t,shift={(0,0)},rotate=0]plot({1.*sin(\t r)},
		{1.*cos(\t r)});
		\filldraw[fill=white, pattern=north east lines]
		(0,1)--(0,4)to[out=90, in=180, looseness=1](0.1,4.1)--(4,4.1)to[out=0, in=90, looseness=1](4.1,4)--(4.1,0.1)to[out=270, in=0, looseness=1](4,0)--(1,0)to[out=90, in=0, looseness=1](0,1);
		\filldraw[fill=white, pattern=north east lines]
		(0,-1)--(0,-4)to[out=270, in=0, looseness=1](-0.1,-4.1)--(-4,-4.1)to[out=180, in=270, looseness=1](-4.1,-4)--(-4.1,-0.1)to[out=90, in=180, looseness=1](-4,0)--(-1,0)to[out=270, in=180, looseness=1](0,-1);
		\path[font=\normalsize]
		(-1,0.5)node[left]{$H$};	
		\end{tikzpicture}\qquad\qquad
		\begin{tikzpicture}[scale=0.8]
		\draw[color=black,scale=1,domain=-3.141: 3.1412,
		smooth,variable=\t,shift={(0,0)},rotate=0]plot({1.*sin(\t r)},
		{1.*cos(\t r)});
		\filldraw[fill=white, pattern=north east lines]
		(0,0)--(0,4)to[out=90, in=180, looseness=1](0.1,4.1)--(4,4.1)to[out=0, in=90, looseness=1](4.1,4)--(4.1,0.1)to[out=270, in=0, looseness=1](4,0)--(0,0);
		\filldraw[fill=white, pattern=north east lines]
		(0,0)--(0,-4)to[out=270, in=0, looseness=1](-0.1,-4.1)--(-4,-4.1)to[out=180, in=270, looseness=1](-4.1,-4)--(-4.1,-0.1)to[out=90, in=180, looseness=1](-4,0)--(0,0);
		\path[font=\normalsize]
		(-1,0.5)node[left]{$E_0$};
		\end{tikzpicture}
	\end{center}
	\caption{Qualitative pictures of the datum $H$ and the minimizer $E_0$.}\label{figinp}
\end{figure}
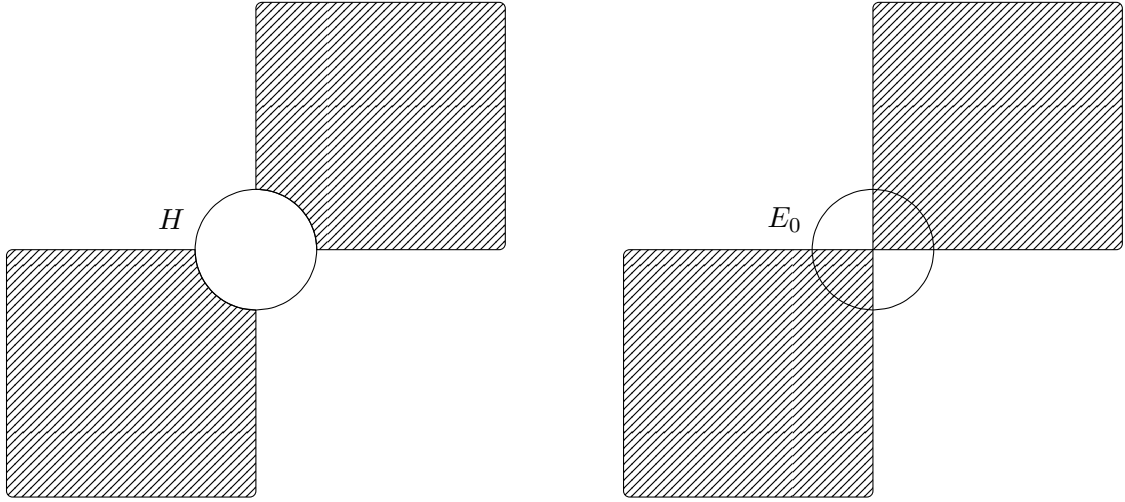

\begin{remark}
	Observe that if $\overline{G}$ is the relaxed functional defined in \cite{BeMu07} recalled in Subsection \ref{BM}, then $\overline{G}(E_0)=+\infty$, and hence $E_0$ will never be detected by a minimization problem \eqref{pbinp} analogously defined with the functional $\overline{G}$.\\
\end{remark}

\noindent We have the following result.
\begin{prop}
	There exists $\la_0\in\big(0,\frac{\pi}{2}\big)$ such that for any $\la\in(0,\la_0)$ the set $E_0$ is the unique minimizer of problem $\fP$.
\end{prop}

\begin{proof}
	Let us first observe that varifolds associated to sets with energy sufficiently close to the infimum of the problem have mass uniformly bounded independently of $\la$. More precisely, suppose that $E$ is a competitor such that $\rflp(E)\le \inf \mathfrak{P} + 1$, and let $\rflp(E)=\flp(V)$ for some $V\in\cA(E)$. Then
	\[
	\int |k_V|^p\,d\mu_V \le 1 + \rflp(E_0) \le 1+ \cF_{\frac\pi2,p}(E_0) \eqqcolon C_1.
	\]
	Using \eqref{ref3} with $\si=1$ and $\ro\to+\infty$ on the varifold $V$ with $x_0=0$ we get
	\[
	\begin{split}
	\mu_V(B_1(0)) 
	&\le - \int_{\R^2\sm B_1} \left\lgl k_V, \frac{x}{|x|} \right\rgl \, d\mu_V(x) - \int_{B_1} \lgl k_V,x\rgl \, d\mu_V(x) \\
	&\le C(H) + \int_{B_1} |k_V| \, d\mu_V \\
	&\le C(H) + C_1^{\frac1p}\mu_V(B_1(0))^{\frac{1}{p'}}.
	\end{split}
	\]
	Hence $\mu_V(\R^2) \le C(H) + \mu_V(B_1(0)) \le \overline{C}=\overline{C}(H,p,E_0)$ and $\overline{C}$ is independent of $\la$.
	
	Now let $E_n$ be a minimizing sequence of problem $\fP$. By Theorem \ref{thmmain} and Lemma \ref{lem2} we can write $\rflp(E_n)=\sum_{i\in I_n} \flp(\ga_{i,n})$ for some curves $\ga_{i,n}$. Up to subsequence $I_n=I$ and the curves converge strongly in $C^1$ and weakly in $W^{2,p}$ to curves $\ga_i$. In particular $E_n\to E$ in the $L^1$ sense, and
	\beq \label{ref12}
		\rflp(E)\le \inf \fP \le \rflp(E_0)=\flp(E_0).
	\eeq
	by lower semicontinuity. Moreover by $C^1$ strong convergence we have that
	\beqs
		\forall i \forall p\in \big((\ga_i)\cap \pa B_1(0)\big)\sm\{(1,0),(0,1),(-1,0),(0,-1)\} \quad \Rightarrow\quad (\ga_i) \mbox{ is tangent to $\pa B_1(0)$ at } p.
	\eeqs
	Observe that $E_0$ carries inside $B_1(0)$ a $\flp$ energy equal to $4\la$.\\
	Arguing as in Subsection \ref{BM}, since $\rflp(E)<+\infty$ there exists a varifold $V=\sum_{i=1}^N (\ga_1)_\sharp(\bv(S^1,1))\in\cA (E)$. Up to renaming and reparametrization assume $\ga_1(0)=(1,0)$, $\ga_1'(0)=-(1,0)$, and $\ga_1|{[-T,0]}$ joins $(0,1)$ and $(1,0)$ having support contained in $\cF H\sm \overline{B_1(0)}$. Since $\ga_1$ is $C^1$ and closed, by the above discussion there exists a first time $\tau>0$ such that $\ga_1$ intersects transversally $\pa B_1(0)$. Also such transversal intersection can take place only at one of the points in $\{(1,0),(0,1),(-1,0),(0,-1)\}$. We divide two cases.\\
	1) If $\ga_1(\tau)\in\{(0,1),(0,-1)\}$, observing that there is $\overline{C}>0$ depending only on the problem $\fP$ such that $L(\ga_i)\le \overline{C}$ for any $i$, then arguing like in \eqref{ineq1} we get
	\beqs
		\flp(\ga_1|_{(0,\tau)}) \ge \la \sqrt{2} + \frac{\pi}{2}\frac{1}{L(\ga_1|_{(0,\tau)})^{\frac{p}{p'}}}\ge \la \sqrt{2} + \frac{\pi}{2}\frac{1}{\overline{C}} >2\la,
	\eeqs
	where the last inequality holds choosing $\la_0$ small enough.\\
	2) If $\ga_1(\tau)=(1,0)$, then by the same argument leading to \eqref{ineq1} one has
	\beq \label{ref13}
		\pi \le \frac{L(\ga_1|_{(0,\tau)})}{p'} + \frac{\cE_p(\ga_1|_{(0,\tau)})}{p}.
	\eeq
	If $\la p'\ge 1$, then $\pi\le \flp(\ga_1|_{(0,\tau)})$. If instead $\la p'<1$, then also $\frac{\la p'}{p}<1$, and multiplying \eqref{ref13} by $\la p'$ one has $\la p' \pi \le \flp(\ga_1|_{(0,\tau)})$. So we can write that $\flp(\ga_1|_{(0,\tau)})\ge \min\{1,\la p'\}\pi$. Choosing $\la_0<\frac{\pi}{2}$ then $\pi> 2\la$, and since $p'>1>\frac{2}{\pi}$ then $\la p' \pi >2\la $; hence in any case
	\beqs
		\flp(\ga_1|_{(0,\tau)})>2\la.
	\eeqs
	By inequality \eqref{ref12} we conclude that $\ga_1(\tau)=(-1,0)$ and $\pa E_0\con \cup_{i=1}^N (\ga_i)$. Hence again by the same inequality we have that $E=E_0$, and thus $\fP$ has a unique minimizer, that is $E_0$.\\
\end{proof}

\textcolor{white}{text}

\subsection{Examples and qualitative properties} \label{subsex}
In this subsection we fix $p\in(1,\infty)$ and we collect some remarks about the qualitative properties of sets $E$ having $\rfp(E)<+\infty$.\\

\noindent First we want to prove a result that is completely analogous to the Theorem 6.5 in \cite{BeDaPa93}. To this aim we need some definitions.

\begin{defn}
	Let $E\con\R^2$ be closed measurable. A point $p\in\pa E$ is called (simple) cusp if there is $r>0$ such that up to rotation and translation the set $B_r(p)\cap\pa E$ is the union of the graphs of two functions $f_1,f_2:[0,a]\to\R$ of class $C^1\cap W^{2,p}$ with $f_i(0)=f_i'(0)=0$, $f_1(x)\le f_2(x)$, and $f_1(x)=f_2(x)$ if and only if $x=0$.
\end{defn}

\noindent Also, we shall need the following definitions in the context of planar graphs.

\begin{defn}
	Let $G\con\R^2$ be a planar finite graph, i.e. a set given by the union of finitely many embeddings of $[0,1]$ of class $C^1\cap W^{2,p}$, called edges of $G$, possibly meeting only at the endpoints, called vertices of $G$. The symbols $E_G,V_G$ respectively denote the set of edges of $G$ and the set of vertices of $G$. Together with the topology of a graph $G$, it is assigned a multiplicity function $m:E_G\to\N$.\\
	For any vertex $v\in V_G$ there is $r_v>0$ such that for $0<r<r_v$ the set $H:=G\cap B_r(v)$ is a finite connected graph whose edges only meet at $v$ and with multiplicity inherited from $G$. In this notation, the local density of $G$ at $v$ is the number $\ro_G(v)=\sum_{e\in E_H} m(e)$.\\
	\noindent Now assume also that for any $v\in V_G$ and $0<r<r_v$, if $f_i$ are regular parametrizations of the edges $e_i$ of the graph $H=G\cap B_r(v)$ with $f_i(1)=v$, then for any $i$ there is $j$ such that the arclength derivatives $\dot f_i,\dot f_j$ satisfy $\dot f_i(1)=-\dot f_j(1)$. Under this assumption, we denote by $w_1(v),...,w_{N_v}(v)$ unit norm vectors identifying the possible tangent directions given by $\{\dot f_i(1)\}_i$. Hence $w_i(v)^\perp$ is the counterclockwise rotation of $w_i(v)$ of an angle equal to $\pi/2$. Finally we define
	\beqs
		I^+(w_i(v)):=\big\{ e_i\in E_H\,\,|\,\,\dot f_i(1)=\pm w_i(v),\,\,\,(\dot f_i(1),w_i(v)) \mbox{ is a negative basis of }\R^2\big\},
	\eeqs
	\beqs
		I^-(w_i(v)):=\big\{ e_i\in E_H\,\,|\,\,\dot f_i(1)=\pm w_i(v),\,\,\,(\dot f_i(1),w_i(v)) \mbox{ is a positive basis of }\R^2\big\},
	\eeqs
	and
	\beqs
		\ro^+_G(v,w_i(v))=\sum_{e_i\in I^+(w_i(v))} m(e_i),
	\eeqs
	\beqs
		\ro^-_G(v,w_i(v))=\sum_{e_i\in I^-(w_i(v))} m(e_i).
	\eeqs
	The graph $G$ is said to be regular if for any $v\in V_G$ and for any $w_i(v)$ it holds that $\ro^+_G(v,w_i(v))=\ro^-_G(v,w_i(v))$.
\end{defn}

\begin{remark}\label{rem2}
	Let $V=\sum_{i=1}^N (\ga_i)_\sharp(\bv(S^1,1))$ be a varifold in $\cA(E)$ for some set $E$. Suppose that $\Ga=\cup(\ga_i)$ is a finite planar graph $G_\Ga$. To any edge $e$ of $G_\Ga$ we assign the multiplicity function $m_\Ga(e)=\te_V(p)$ for any $p\in e$ that is not a vertex. By the flux property (Definition \ref{defflux}), the graph $G_\Ga$ with the multiplicity $m_\Ga$ is regular.
\end{remark}

\noindent We are ready to prove the following result about the energy of sets which are smooth out of finitely many cusp points. The strategy follows ideas from \cite{BeDaPa93}, but it is different in the technical parts.

\begin{thm}\label{thmcusps}
	Let $E\con\R^2$ be a closed set whose boundary is $C^1\cap W^{2,p}$ smooth at every point but at finitely many ones which are simple cusps $q_1,...,q_k$. Then
	\beq
		\overline{\cF_p}(E)<+\infty \qquad\Leftrightarrow\qquad k\mbox{ is even.}
	\eeq
\end{thm}

\begin{proof}
	If $k$ is even, Theorem 6.4 in \cite{BeDaPa93} implies that the relaxed energy $\overline{G}(E)$ studied in \cite{BeDaPa93} is finite. Since $\overline{\cF_p}\le \overline{G}$, we have one implication.\\
	Now suppose that $\overline{\cF_p}(E)$ is finite, i.e. $\cA(E)\neq\epty$. Let $V=\bv(\Ga,\te_V)=\sum_{i=1}^N(\ga_i)_\sharp(\bv(S^1,1))\in\cA(E)$. We are going to construct a set $\tE$ satisfying the hypotheses of the theorem and having the same unknown number of cusps of $E$, together with a varifold $\tV=\bv(\tGa,\tilde{\te}_{\tV})\in\cA(\tE)$ with the additional property that $\tGa$ is a finite graph $G_{\tGa}$ with multiplicity as given in Remark \ref{rem2}. Once the support of a varifold in $\cA(\tE)$ is a finite graph, we can prove that the number of cusps is even.\\
	
	\noindent \emph{Step 1.} Now we construct $\tE$ and $\tGa$ as claimed. Let $\cC(\Ga)$ be the set of points $p\in\Ga$ such that in any neighborhood of $p$ it is impossible to write $\Ga$ as a single graph, i.e. $p$ is a crossing or a branching point of two pieces of some curves $\ga_i,\ga_j$. Call $K$ the set of accumulation points of $\cC(\Ga)$. Observe that $K$ is compact.\\
	Also, observe that if a sequence $p_n\in\cC(\Ga)$ converges to $\bar{p}$, $p_n=\ga_i(t_n)=\ga_j(\tau_n)$ with $t_n\to t^-,\tau_n\to\tau^-$ or $t_n\to t^-,\tau_n\to\tau^+$ and $t\neq \tau$, then $\dot\ga_i(t)=\pm \dot\ga_j(\tau)$.\\
	Now fix $\ep<<1$ and let $q\in K$. Let $v_1(q),...,v_{N_q}(q)$ be unit vectors identifying the tangent directions at $q$ of the curves passing through $q$. For $j=1,...,N_q$ let $\si_1^j,...,\si_{M_{q,j}}^j$ be suitable restrictions of the curves $\{\ga_i\}$ on disjoint intervals $I_i^j=\mbox{domain}(\si_i^j)$ such that each $\si_i^j$ passes through $q$ with tangent parallel to $v_j(q)$. Also, for $i=1,...,N_q$ let $R_i(q)$ be open rectangles with two sides parallel to $v_i(q)$. Up to restriction we assume that each $\si_i^j$ is contained in $\overline{R_j(q)}$ with endpoints on the boundary of the rectangle. We can assume the following properties:\\
	i) each rectangle contains at most one cusp and cusps do not lie on the boundary of any rectangle. Also if $q\in\Ga\sm\pa E$, then $\overline{R_i(q)}\cap\pa E=\epty$;\\
	ii) $R_i(q)\cap \pa E$ is homeomorphic to a closed segment such that: if no cusps lie in $R_i(q)$ then $R_i(q)\cap \pa E$ is the graph of a $C^1\cap W^{2,p}$ function, if a cusp lies in $R_i(q)$ then $R_i(q)\cap \pa E$ is the union of the graphs of two $C^1\cap W^{2,p}$ functions as in the definition of simple cusp;\\
	iii) each $\si_i^j$ can be parametrized as graph inside $R_j(q)$, and $|\dot \si_i^j(\cdot)+ v_j(q)|\le\ep$ or $|\dot \si_i^j(\cdot)- v_j(q)|\le\ep$;\\
	iv) $\si_i^j$ intersects $\pa R_j(q)$ only on the sides perpendicular to $v_j(q)$ and transversely, and $\si_i^j$ intersects $\si_k^l$ only in the open set $R_j(q)\cup R_l(q)\sm(\pa R_j(q)\cup \pa R_k(q))$;\\
	v) if $a\in\pa I_i^j,b\in\pa I_k^j$ and $\si_i^j(a)=\si_i^k(b)$, then $\dot \si_i^j(a)=\pm \dot \si_k^j(b)$;\\
	vi) if $\si_i^j(a)\in\cF E$, then $\te_V(\si_i^j(a))=\sharp\{k\,\,|\,\,\si_k^j \mbox{ passes through } \si_i^j(a) \}$ is odd; if $\si_i^j(a)\in\Ga\sm\pa E$, then $\te_V(\si_i^j(a))=\sharp\{k\,\,|\,\,\si_k^j \mbox{ passes through } \si_i^j(a) \}$ is even.\\
	Property v) follows by the fact that transverse crossings of two curves are at most countable (as proved in Lemma \ref{lem2}), and property vi) follows from the fact that $V\in\cA(E)$ and thus $\te_V$ is odd (resp. even) at $\cH^1$-ae point of $\cF E$ (resp. $\Ga\sm\pa E$).\\
	Since the set $K$ is compact, we can extract a finite covering of rectangles corresponding to points $q_1,...,q_L$. By Theorem \ref{propmon} the numbers $N_{q_i}$ of the rectangles of $q_i$ are uniformly bounded in terms of the energy, which is finite. Hence we can add to the cover the possibly remaining rectangles corresponding to each $q_i$, yielding a covering that is still finite. For any $j=1,...,L$ and $i=1,...,N_{q_j}$ we are going to modify the curves $\si_i^j$ in a finite number of steps. We start from the family $\{\si_i^1\}_{i=1}^{M_{q_1,1}}$ corresponding to $R_1(q_1)$, then one modifies the curves corresponding to $R_2(q_1)$ and so on up to $R_{N_{q_1}}(q_1)$, then one changes the curves of the families corresponding to $q_2$ and so on up to $q_L$. Since the procedure is the same at any step, let us describe only the case of the family $\{\si_i^1\}_{i=1}^{M_{q_1,1}}$ corresponding to $R_1(q_1)$. In the end we will end up with the desired $\tE,\tGa$.\\
	We modify a $\si_i^1$ as follows, depending on the cases $q_1\in \Ga\sm\pa E$, or $q_1\in \cF E$, or $q_1$ is a cusp.\\
	1) Suppose $q_1\in \Ga\sm\pa E$. Fix $\si_i^1$ and split it into the two pieces divided by $q_1$. Let us say that one such piece of $\si_i^1$ is parametrized as graph by $f:[0,\al]\to\R$ with $f(0)=f'(0)=0$ corresponding to $q_1$. Let $u_i^1$ be the solution of
	\beq\label{eq5}
		\begin{cases}
			u(x)=\la x^3+\mu x^2+\nu x+\om,\\
			u(0)=u'(0)=0,\\
			u(\al)=f(\al),\,\,\,u'(\al)=f'(\al),
		\end{cases}
	\eeq
	for the suitable constants $\la, \mu, \nu, \om$. Doing the same with the other piece of $\si_i^1$, we substitute each $\si_i^1$ with the graphs of the obtained functions $u_i^1$ (such modification is then a change in one of the original curves $\ga_i$'s). Observe that by properties v), vi) one obtains a new varifold still in the class $\cA(E)$, in fact graphs of finitely many polynomials meet in at most finitely many points.\\
	2) Suppose now that $q\in\cF E$. By construction, for example $R_1(q_1)$ contains some curves with endpoints on $\cF E\cap \pa R_1(q_1)$. In this case we modify the curves exactly as before following the system \eqref{eq5}; moreover we declare that the boundary $\cF E$ is modified inside $R_1(q_1)$ following the new modified curves having endpoints on $\cF E\cap \pa R_1(q_1)$. This leads to a new set which we already call $\tE$ satisfying the hypotheses of the theorem and having the same number of cusps of $E$, together with a new varifold already called $\tV$ in the class $\cA(\tE)$ (as before by properties v), vi), together with the fact that the new curves are graphs of polynomials).\\
	3) Finally suppose $q_1$ is a cusp of $\pa E$. In this case we modify the curves (and the set $E$) exactly in the same way of the case 2). This preserves the cusp in the new set $\tE$.\\
	After performing these modifications in any $R_i(q_j)$ we end up with a varifold $\tV$ given by curves $\tilde{\ga}_i$ such that the set $\cC(\tGa)$ of the points $p\in\tGa$ such that in any neighborhood of $p$ it is impossible to write $\tGa$ as a single graph is finite. In fact the points of this type belonging to the union of the closure of the rectangles $R_i(q_j)$ are finite. So, if by contradiction there are points of $\cC(\tGa)$ accumulating to some limit point $q$, this would be outside the union of the rectangles $R_i(q_j)$, and $q$ would be a limit of a sequence in $\cC(\Ga)$. Hence $q$ would be in $K$, and thus in the interior of some rectangle $R_i(q_j)$, that is a contradiction.\\
	
	\noindent \emph{Step 2.} Now we show that, in general, if a set $E$ is as in the hypotheses of the theorem and if $V=\bv(\Ga,\te_V)\in\cA(E)$ is such that $\Ga$ is a finite graph, then the number of cusps of $E$ is even. Together with Step 1, this gives the conclusion. Here we essentially generalize the strategy of \cite{BeDaPa93}.\\
	Call $G_\Ga$ the finite graph given by $\Ga$ with multiplicity $m_\Ga$ as described in Remark \ref{rem2} (recall that $G_\Ga$ is regular). Let us construct a new graph $G$ with multiplicity $m$ as follows. If $e\in E_{G_\Ga}$, then define the multiplicity
	\beqs
		m(e):=\begin{cases}
			\frac{m_\Ga(e)}{2} & \mbox{ if } m_\Ga(e) \mbox{ even,}\\
			\frac{m_\Ga(e)-1}{2} & \mbox{ if } m_\Ga(e) \mbox{ odd,}
		\end{cases}
	\eeqs
	with the convention that if $m(e)=0$, then the edge $e$ does not appear in $G$. Now let $y\in V_G$. We want to evaluate the parity of $\ro_G(y)$ dividing some cases.\\
	a) Suppose $y\not\in\pa E$. Then any edge $e$ of $G_\Ga$ with endpoint at $y$ has $\ro^+_{G_\Ga}(y,w_i(y))=\ro^-_{G_\Ga}(y,w_i(y))$ even for any $w_i(y)$. Hence by definition we have that $\ro_G(y)$ is even.\\
	b) Suppose $y\in\cF E$. Then exactly two edges $e_1,e_2$ of $G_\Ga$ having an endpoint at $y$ have odd multiplicity: $m_\Ga(e_i)=2k_i+1$ for $i=1,2$. Up to relabeling suppose that $e_1\in I^+(w_1(y))$ and $e_2\in I^-(w_1(y))$. Every other edge of $G_\Ga$ having an endpoint at $y$ has even multiplicity. Since $G_\Ga$ is regular we have that
	\beqs
		2k_1+1+2a^+_1=\ro^+_{G_\Ga}(y,w_1(y))=\ro^-_{G_\Ga}(y,w_1(y))=2k_2+1+2a^-_1,
	\eeqs
	and similarly
	\beqs
		2a^+_i=\ro^+_{G_\Ga}(y,w_i(y))=\ro^-_{G_\Ga}(y,w_i(y))=2a^-_i,
	\eeqs
	for any possible $i\ge2$. Then
	\beqs
		\ro_G(y)=k_1+a^+_1+k_2+a^-_1 +\sum_{i\ge2} a^+_i+a^-_i =2\bigg(k_1+a^+_1+\sum_{i\ge2} a^+_i\bigg)
	\eeqs
	is even.\\
	c) Finally suppose that $y$ is a cusp of $\pa E$. Then exactly two edges $e_1,e_2$ of $G_\Ga$ having an endpoint at $y$ have odd multiplicity: $m_\Ga(e_i)=2k_i+1$ for $i=1,2$. Up to relabeling suppose that $e_1,e_2\in I^+(w_1(y))$. Every other edge of $G_\Ga$ having an endpoint at $y$ has even multiplicity. Since $G_\Ga$ is regular we have that
	\beqs
		2k_1+1+2k_2+1+2a^+_1=\ro^+_{G_\Ga}(y,w_1(y))=\ro^-_{G_\Ga}(y,w_1(y))=2a^-_1,
	\eeqs
	and similarly
	\beqs
		2a^+_i=\ro^+_{G_\Ga}(y,w_i(y))=\ro^-_{G_\Ga}(y,w_i(y))=2a^-_i,
	\eeqs
	for any possible $i\ge2$. Then
	\beqs
		\ro_G(y)=k_1+k_2+a^+_1+a^-_1 + \sum_{i\ge2} a^+_i+a^-_i= 2(k_1+k_2+a^+_1)+1+2\sum_{i\ge2} a^+_i,
	\eeqs
	that is odd.\\
	It follows that the cusps of $\pa E$ coincides with the vertices $y$ of $G$ having odd local density $\ro_G(y)$. By Theorem 1.2.1 in \cite{Or62}, the vertices of a finite graph with odd local density are even. Hence the cusps are even and the proof is completed.
\end{proof}

\noindent Now we turn our attention to another class of sets. Let us give the following definition.

\begin{defn} \label{defpoligoni}
	A closed measurable set $E\con\R^2$ is a $p$-polygon if $\pa E=(\ga)$ for a curve $\ga:[0,2\pi]/_\sim \simeq S^1\to\R^2$ such that:\\
	i) $\ga$ is injective,\\
	ii) there exist finitely many times $t_1<t_2<...<t_K$ such that $\ga|_{(t_i,t_{i+1})}\in W^{2,p}$ for $i=1,...,K$ (with $t_{K+1}=t_1$), and $\ga'(t_i^-),\ga'(t_i^+)$ are linearly independent for $i=1,...,K$.
\end{defn}

\begin{prop} \label{poligoni}
	Let $E$ be a $p$-polygon, then $\rfp(E)=+\infty$.
\end{prop}

\begin{proof}
	Let $\ga$ be as in the definition of $p$-polygon. Without loss of generality we can assume that $0=\ga(0)$ is such that $\ga'(0^-)$ and $\ga'(0^+)$ are linearly independent. Suppose by contradiction that there is a varifold $V=\bv(\Ga,\te_V)=\sum_{i=1}^N (\ga_i)_\sharp(\bv(S^1,1))\in \cA(E)$. Let $v=\ga'(0^-)$, then since $V$ verifies the flux property we find a nice rectangle $R_v(p)$ at $p$ with side parameters $a,b$ for the curves $\{g_j\}_{j=1}^r$ given by the definition of flux property. We can suppose that $g_1|_{[-\ep,0)}\con \cF E$, $g_1|_{(0,\ep]}\con \Ga \sm \pa E$, and that $(g_i)\cap \pa E = \{0\}$ for $i=2,...,r$. Hence
	\beqs
	g_1|_{[-\ep,0)}\con\{\te_V \mbox{ odd}\},
	\eeqs
	\beqs
	\cH^1\bigg(\bigg(\bcup_{i=2}^r (g_i) \cup g_1\big((0,\ep]\big) \bigg)\sm \{\te_V \mbox{ even}\} \bigg) =0.
	\eeqs
	Then there exists $c_1\in (-a,0)$ such that
	\beqs
	\sum_{z\in \cup_{j=1}^r(g_i) \cap \{y\,|\,\lgl y-p,v\rgl=c_1\}} \te_V(z)= M_1
	\eeqs
	with $M_1$ odd, and there exists $c_2\in (0,a)$ such that
	\beqs
	\sum_{z\in \cup_{j=1}^r(g_i) \cap \{y\,|\,\lgl y-p,v\rgl=c_2\}} \te_V(z)= M_2
	\eeqs
	with $M_2$ even. But by the flux property $M_1$ and $M_2$ should be equal, thus we have a contradiction.
\end{proof}

\begin{remark}
	More generally it follows from the proof of Proposition \ref{poligoni} that roughly speaking $\rfp(E)=+\infty$ whenever the boundary $\pa E$ has an angle (in the same sense of the definition of polygon).\\
\end{remark}

\noindent With the strategy in the proof of Proposition \ref{poligoni} we can construct an example of a set $E\con\R^2$ such that $E$ is a set of finite perimeter such that the associated varifold $V_E$ verifies that
\beqs
	\si_{V_E}=0, \qquad k_{V_E}\in L^2(\mu_{V_E}), \qquad\mbox{but  } \rfp(E)=+\infty.
\eeqs
Such set is discussed in the next example.

\begin{example} \label{ex1}
	Consider a positive angle $\te>0$ which will be taken very small and the vectors in the plane identified by the complex numbers
	\beq \label{ref16}
		e^{-i\te},\qquad e^{-i2\te},\qquad e^{i(-\pi+\te)},\qquad e^{i(-\pi+2\te)}.
	\eeq
	The sum of such vectors gives the point $(0, -2(\sen(\te)+\sen(2\te)))$. Now let $\vp>0$ be another positive angle and consider the vectors
	\beq \label{ref17}
		e^{i\vp},\qquad e^{i(\pi-\vp)},
	\eeq
	so that the sum of these last vectors gives the point $(0, 2\sin(\vp))$. Then for $\te\to0$, since $\sen(\te)+\sen(2\te)= 3\te + o(\te^2)$ there exists $\vp=3\te+o(\te^2)$ such that the sum of the vectors in \eqref{ref16} and \eqref{ref17} is zero.\\
	Given these vectors we can define a set $E$ as in Fig. \ref{figangoli} whose boundary is the image of three smooth closed immersions $\si_i$ of the interval $[0,1]$ having $\si_i(0)=\si_i(1)=0$ with derivative $\si'_i(0),\si'_i(1)$ proportional to the vectors in \eqref{ref16}, \eqref{ref17}. In such a way the varifold $V_E$ clearly verifies that $\si_{V_E}=0$ and $k_{V_E}\in L^2(\mu_{V_E})$. However arguing as in the proof of Proposition \ref{poligoni} and assuming $\rfp(E)<+\infty$, one immediately gets a contradiction. Hence $\rfp(E)=+\infty$.\\
	\begin{figure}[h]
		\begin{center}
			\begin{tikzpicture}[scale=2.5]
			\draw
			(-1.5,0)--(2.5,0);
			\draw
			(0,0)--(0.94,0.342);
			\filldraw[fill=white, pattern=north east lines]
			(0,0)to[out=20, in=0, looseness=1](0,1)to[out=180, in=150, looseness=1](0,0);
			\draw
			(0,0)--(1.4895,-0.1725);
			\draw
			(0,0)--(1.954, -0.422);
			\filldraw[fill=white, pattern=north east lines]
			(0,0)--(0.993,-0.115)to[out=-6.6, in=-12.2, looseness=1](0.977, -0.211)--(0,0);
			\filldraw[fill=white, pattern=north east lines]
			(0,0)--(-0.993,-0.115)to[out=-173.4, in=-167.8, looseness=1](-0.977, -0.211)--(0,0);
			\path[font=\normalsize]
			(-0.6,0.7)node[left]{$E$};
			\draw[color=black,scale=1,domain=0: 0.349,
			smooth,variable=\t,shift={(0,0)},rotate=-70]plot({1.*sin(\t r)},
			{1.*cos(\t r)});
			\draw[color=black,scale=1,domain=0: 0.1164,
			smooth,variable=\t,shift={(0,0)},rotate=-90]plot({1.5*sin(\t r)},
			{1.5*cos(\t r)});
			\draw[color=black,scale=1,domain=0: 0.21,
			smooth,variable=\t,shift={(0,0)},rotate=-90]plot({2*sin(\t r)},
			{2*cos(\t r)});
			\path[font=\normalsize]
			(1.2,0.2)node[left]{$\vp$};
			\path[font=\normalsize]
			(1.7,-0.1)node[left]{$\te$};
			\path[font=\normalsize]
			(2.25,-0.25)node[left]{$2\te$};
			\end{tikzpicture}
		\end{center}
		\caption{Picture describing the set $E$ of Example \ref{ex1}. The set is symmetric with respect to the reflection about the vertical axis.}\label{figangoli}
	\end{figure}
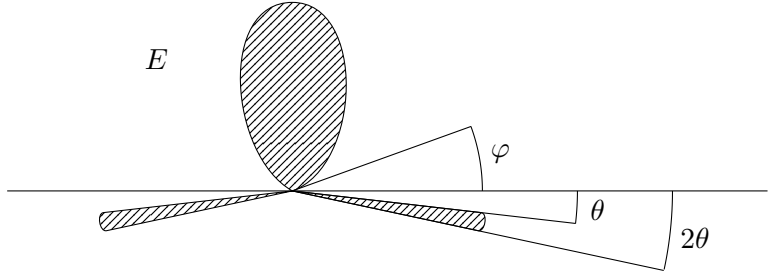
\end{example}

\noindent Finally we construct a simple example showing that there are sets $E$ with $\rfp(E)<\infty$, but such that $\cH^1(\pa E\sm \cF E)>0$ and $\pa E$ is the support of a $C^\infty$ immersion $\si$.\\

\begin{example}\label{example}
	Let us construct a set $E$ such that $\pa E= (\ga)$ for a $C^\infty$ immersion $\ga:S^1\to\R^2$, $\cH^1(\pa E\sm \cF E)>0$, and $\rfp(E)<+\infty$.\\
	Let $\{q_n\}_{n\ge 1}=\Q\cap [0,1]$ be an enumeration of the rationals in $[0,1]$, and define $K=[0,1]\sm \cup_{n\ge 1} (q_n-2^{-n-2},q_n-2^{-n-2})$. The set $K$ is compact and $\cL^1(K)\ge 1- \sum_{n=1}^{\infty} 2^{-n-1}= \frac{1}{2}$. Consider a $C^\infty$ nonincreasing function $\vp:[0,\infty)\to [0,1]$ such that $\vp(0)=1$, $\vp(t)=0$ for $t\ge 1$ and let
	\beqs
	f(x)=\sum_{n=1}^{\infty} \frac{1}{2^n} \vp\bigg( \frac{(x-q_n)^2}{(2^{n+2})^2} \bigg) \qquad\forall x\in[0,1].
	\eeqs
	By construction we have that $K=f^{-1}(0)$. Moreover $f\in C^\infty([0,1])$, in fact $\vp\le 1$ and $|\vp^{(k)}|\le c_k$ for any $k\ge 1$ for some $c_k>0$, so that both the series $f$ and the series of the derivatives totally converge. Then we can define a $C^\infty$ parametrization $\si:[0,4]\to\R^2$ such that $\si(t)=(t,f(t))$ for $t\in[0,1]$, $\si(t)=(3-t,-f(t))$ for $t\in[2,3]$, while $\si|_{[1,2]}$ and $\si|_{[3,4]}$ parametrize two drops with vertices respectively at $(1,0)$ and $(0,0)$. Therefore $\si$ parametrizes the boundary of a bounded set $E$ which is the planar surface enclosed by the two drops and lying between the graphs of $f$ and $-f$.\\
	By construction $\pa E=(\si)$ and $\cF E= (\si)\sm K$, hence $\cH^1(\pa E\sm \cF E)\ge \frac{1}{2}$. However approximating $f$ with $f_n(x)=f(x)+\frac{1}{n}\psi(x)$, where $\psi\in C^\infty([0,1];[0,1])$ is such that $\psi(0)=\psi(1)=0$, $\psi|_{(0,1)}>0$, and defining $\si_n$ in analogy with $\si$, we conclude that $\rfp(E)<+\infty$.\\
\end{example}

\begin{figure}[h]
	\begin{center}
		\begin{tikzpicture}[scale=2.5]
		\filldraw[fill=white, pattern=north east lines]
		(0,0)to[out=90, in=0, looseness=1](-0.5,0.5)to[out=180, in=135, looseness=1](-1,0)to[out=45, in=135, looseness=1](0,0);
		\filldraw[fill=white, pattern=north east lines]
		(0,0)to[out=270, in=0, looseness=1](-0.5,-0.5)to[out=180, in=-135, looseness=1](-1,0)to[out=-45, in=-135, looseness=1](0,0);
		\filldraw[fill=white, pattern=north east lines]
		(0,0)to[out=45, in=90, looseness=1](1,0)to[out=-90, in=-45, looseness=1](0,0);
		\path[font=\Huge]
		(0.098,-0.007)node[left]{$\cdot$};
		\path[font=\Huge]
		(-0.9,0)node[left]{$\cdot$};
		\path[font=\normalsize]
		(0.3,0.3)node[left]{$E$};
		\path[font=\normalsize]
		(-1,0)node[left]{$x$};
		\path[font=\normalsize]
		(0,0)node[left]{$y$};
		\end{tikzpicture}
	\end{center}
	\caption{An example of a set of finite perimeter $E$ such that $\cF_p(E)=\cF_p(V)<+\infty$ for any $p\in[1,\infty)$, where $V\in\cA(E)$ is the varifold induced by a smooth immersion $\ga$ parametrizing $\pa E$. Here $\pa E= \cF E\sqcup\{x,y\}$ and the strict inclusions $\cF E \subsetneq \{x\,\,|\,\,\te_V(x) \mbox{ is odd}\}=\cF E \sqcup \{y\}\subsetneq \pa E$ occur.}
\end{figure}
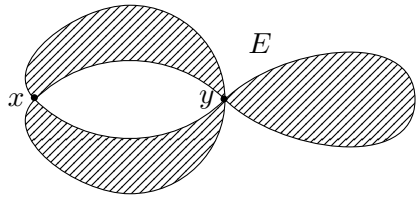
%\vspace{1cm}

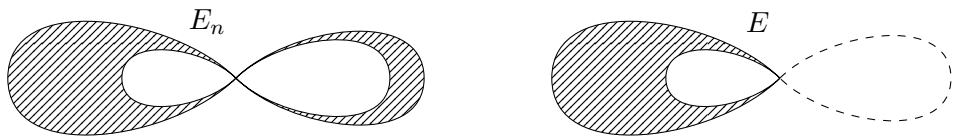
\begin{figure}[h]
	\begin{center}
		\begin{tikzpicture}[scale=1.5, rotate=180]
		\filldraw[fill=white, pattern=north east lines]
		(0,0)to[out=45, in=90, looseness=1](1,0)to[out=270, in=-45, looseness=1](0,0)to[out=-45, in=270, looseness=1](2,0)to[out=90, in=45, looseness=1](0,0);
		\filldraw[fill=white, pattern=north east lines, rotate=180]
		(0,0)to[out=45, in=90, looseness=1](1.35,0)to[out=270, in=-45, looseness=1](0,0)to[out=-45, in=270, looseness=1](1.65,0)to[out=90, in=45, looseness=1](0,0);
		\path[font=\normalsize]
		(0,-0.5)node[left]{$E_n$};
		\end{tikzpicture}
			\qquad\qquad
		\begin{tikzpicture}[scale=1.5]
		\draw[dashed]
		(0,0)to[out= 45,in=90, looseness=1] (1.5,0)to[out=270 ,in=-45, looseness=1] (0,0);
		\filldraw[fill=white, pattern=north east lines, rotate=180]
		(0,0)to[out=45, in=90, looseness=1](1,0)to[out=270, in=-45, looseness=1](0,0)to[out=-45, in=270, looseness=1](2,0)to[out=90, in=45, looseness=1](0,0);
		\path[font=\normalsize]
		(0,0.5)node[left]{$E$};
		\end{tikzpicture}
	\end{center}
	\caption{An example of a set $E$ with finite relaxed energy such that $\pa E\sm \cF E$ is a singleton. A sequence of sets $E_n$ converging to $E$ with uniformly bounded energy is for example made of sets like in the one on the left in the picture; the dashed line represents the corresponding ghost line given by the collapsing of the right part of the sets $E_n$.}
\end{figure}

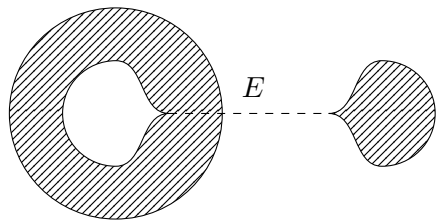
\begin{figure}[H]
	\begin{center}
		\begin{tikzpicture}[scale=0.7, rotate=90]
		\filldraw[fill=white, pattern=north east lines, even odd rule]
		(0,2) circle (2cm)  (0,3)to[out=0, in=90](1,2)to[out=270, in=90](0,1)to[out=90, in=270](-1,2)to[out=90, in=180](0,3);
		\filldraw[fill=white, pattern=north east lines, even odd rule]
		(0,-2)to[out=270, in=90](1,-3)to[out=270, in=0](0,-4)to[out=180, in=270](-1,-3)to[out=90, in=270](0,-2);
		\draw[dashed]
		(0,1)--(0,-2);
		\path[font=\normalsize]
		(0.5,-1)node[left]{$E$};
		\end{tikzpicture}
	\end{center}
	\caption{An example of a set $E$ with finite relaxed energy such that, by Lemma \ref{lemalternativa}, the multiplicity $\te_V$ is not locally constant on connected components of $\cF E$.}
\end{figure}
\vspace{1.5cm}

\textcolor{white}{text}

\noindent \emph{Acknowledgments.} I warmly thank Matteo Novaga for having proposed me to study this problem and for many useful conversations. I also thank the referee for suggesting an improvement on the work.\\

%%%%%%%%%%%%%%%%%%%%%%%%%%%%%%%%%%%%%%%%%%%%%%%%%%%%%%%%%%%%%%%%%%%%%%%%%%%%%%
%BIBLIOGRAFIA

\end{document}